\def\standalonechapter{}

\ifdefined\standalonechapter
\documentclass[11pt]{article}

\usepackage[utf8]{inputenc}
\usepackage[T1]{fontenc}
\usepackage[english]{babel}

\usepackage{amsmath, amssymb, amsthm}
\usepackage{mathtools}
\usepackage{mathrsfs}
\usepackage{bbm}

\usepackage{geometry}
\usepackage[
    colorlinks=false,
    pdfborder={0 0 1}
]{hyperref}
\usepackage{enumitem}

\usepackage{booktabs}

\theoremstyle{plain}
\newtheorem{theorem}{Theorem}[section]
\newtheorem{lemma}[theorem]{Lemma}

\newtheorem{corollary}[theorem]{Corollary}

\theoremstyle{definition}
\newtheorem{definition}[theorem]{Definition}
\newtheorem{remark}[theorem]{Remark}

\numberwithin{equation}{section}

\newcommand{\R}{\mathbb{R}}

\begin{document}

\begin{flushright}
\begin{minipage}{0.55\textwidth} % can vary 0.5–0.6
\small\itshape
This is the second of ten papers devoted to reflections on the Millennium Problem. 
It generalizes known results on the 3D Navier–Stokes equations based on previous studies. 
We hope that the presented material will be useful, and we would be grateful for attention, verification, 
and participation in the further development of the topic.
\end{minipage}
\end{flushright}

\vspace{1.5em}

\begin{center}
    {\LARGE\bfseries {Log-free estimate of the full nonlinearity in the three-dimensional Navier–Stokes equations outside the diagonal regime}\par}
    \vspace{1em}
    {\large\bfseries Pylyp Cherevan\par}
\end{center}

\vspace{1em}

{\small\noindent\textbf{Abstract.}
We investigate the contribution of the full nonlinearity outside the narrow diagonal zone in the three-dimensional Navier–Stokes equations. We consider the off–diagonal components, including $\ell h$, $h\ell$, as well as part of the resonant block $hh\to\ell$ for $|\xi+\eta|\gtrsim N^{1-\delta}$.

The proof relies on three main elements:
\begin{itemize}[label=--]
  \item six-fold integration by parts in the phase $\Phi(t,x,\xi,\eta)=x\!\cdot\!(\xi+\eta)+4t\rho_1\rho_2$ with respect to $(t,\rho_1,\rho_2)$; on the window $|t|\lesssim N^{-1/2}$ the phase Hessian $A=\nabla^2_{(t,\rho_1,\rho_2)}\Phi$ is non-degenerate and provides a reserve $|{\det A}|\sim N^{3/2-\delta}$;
  \item local Strichartz estimates on cylinders of scale $N^{-1/2}$; in §4 a strengthened version is used to combine with the decoupling scheme, while the unconditional framework is based on heat reduction (§\ref{app:duhamel-phase}) and globalization (§\ref{sec:clarif});
  \item bilinear $\varepsilon$–free decoupling in folded geometry of rank~4 (Appendix~\ref{app:decoupling}), yielding a gain of $N^{-1/4}$ for angular tiles of width $N^{-1/2}$.
\end{itemize}

For the narrow corona, suppression of the null–form type symbol is realized when $\delta>1/2$; for the block $hh\to h$ with output projection $P_N$ this mechanism is not required and is accounted for separately (see §~\ref{sec:clarif6}). The combined count yields an a priori estimate without logarithmic losses in the norm $L^1_t\dot H^{-1}_x$ over the whole zone $|\xi+\eta|\gtrsim N^{1-\delta}$ for
$\delta\in\bigl(\tfrac{1}{3},\tfrac{5}{8}\bigr]$; the upper bound is imposed by the stability of the phase reserve $|{\det A}|\sim N^{3/2-\delta}\gg1$ on the window $|t|\lesssim N^{-1/2}$. The full scheme and navigation through the sections are given in the text.
}

% --- Contents ---
\tableofcontents
\fi

% --- Sections of the paper ---

\newpage

\section*{Introduction}
\addcontentsline{toc}{section}{Introduction}
\label{sec:intro}

We consider the three-dimensional Navier--Stokes equation
\begin{equation}
\label{eq:NS}
\partial_t u - \Delta u + (u \cdot \nabla) u + \nabla p = 0, \qquad \nabla \cdot u = 0,
\end{equation}
for a divergence-free vector field $u = u(t,x): [0,T] \times \mathbb{R}^3 \to \mathbb{R}^3$. The scaling symmetry
\begin{equation}
\label{eq:scaling}
u_\lambda(t,x) := \lambda u(\lambda^2 t, \lambda x)
\end{equation}
leaves the norm $\|u\|_{\dot H^{1/2}_x}$ invariant, so the space $L^\infty_t \dot H^{1/2}_x$ is considered critical for~\eqref{eq:NS}.

\medskip
\noindent
\textbf{Goal of the work.} The main focus is the analysis of the contribution of the full nonlinearity at the $\dot H^{-1}_x$ level, corresponding to the right-hand side of the equation. Particular attention is given to the \textit{off--diagonal} regime, in which the frequencies $\xi$, $\eta$ in the bilinear structure $u \otimes u$ are located away from the diagonal $\xi + \eta \approx 0$.

\begin{definition}[off--diagonal zone]\label{def:offdiag-zone}
Let \(N=2^k\) be a fixed dyadic number and \(\delta\in(\tfrac13,\tfrac58]\).
Define the set
\begin{equation}
  \mathcal O_N
  := \bigl\{(\xi,\eta)\in\mathbb{R}^3\times\mathbb{R}^3:\ |\xi|\sim|\eta|\sim N,\ \ |\xi+\eta|\ge N^{1-\delta}\bigr\}.
  \label{eq:offdiag-mask}
\end{equation}
\end{definition}

In what follows the parameter \(\delta\) varies in the range~\eqref{eq:delta-range}. On \(\mathcal O_N\) the phase matrix is non-degenerate, sufficient for six-fold integration by parts; see §\ref{sec:ibp_phase}.

\noindent
\textbf{Main result.} Under the regularity assumption
\begin{equation}
\label{cond:regularity}
u \in L^\infty_t \dot H^{1/2}_x \cap L^2_t \dot H^1_x,
\end{equation}
consistent with scaling invariance, we obtain the following a priori estimate.

\begin{theorem}[log-free estimate outside the diagonal]
\label{thm:offdiag-main}
Let $u \in C^\infty([0,T] \times \mathbb{R}^3)$ be a divergence-free solution of~\eqref{eq:NS}, satisfying~\eqref{cond:regularity}. Then for any $N = 2^k$ and $\delta \in (\tfrac13,\tfrac{5}{8}]$ one has:
\begin{equation}
\label{eq:offdiag-bound}
\left\| \mathcal{N}_N^{\mathrm{off}}(u,u) \right\|_{L^1_t \dot H^{-1}_x}
\lesssim N^{-1} \cdot \|u\|_{L^\infty_t \dot H^{1/2}_x} \cdot \|u\|_{L^2_t \dot H^1_x},
\end{equation}
where
\[
\mathcal{N}_N^{\mathrm{off}}(u,u) := P_N \mathbb{P} \nabla \cdot \left[ P_{\sim N} u \otimes P_{\sim N} u \cdot \mathbf{1}_{\mathcal{O}_N} \right],
\quad P_{\sim N} := \sum_{|\log_2 M - \log_2 N| \le 2} P_M.
\]
The estimate matches the expected scaling order $N^{-1}$ in the $\dot H^{-1}_x$ norm.
\end{theorem}

\noindent
\textbf{Remark.} The series $\sum_{N \in 2^{\mathbb{Z}}} N^{-1} \dots$ converges geometrically, so the resulting estimate over the entire zone $|\xi + \eta| \gtrsim N^{1 - \delta}$ contains no logarithmic defect. The upper bound \(\delta\le \tfrac58\) is consistent with the requirement of stability of the phase reserve \(|\det A|\sim N^{3/2-\delta}\gg1\) on windows \(|t|\lesssim N^{-1/2}\). Note also that the \emph{narrow corona and null--form module} is active only when \(\delta>\tfrac12\) and is used separately for coronal outputs (see below).

\medskip
\noindent\emph{Organization of the proof.}
We conduct the proof along two parallel lines: (A) conditional (local $L^6$ hypothesis \eqref{eq:aux-Stri} + $\varepsilon$-free decoupling of rank~4) and (B) unconditional “heat” line. All global conclusions of the theorem already follow from line (B); line (A) only strengthens the local balance. See also the module map in Appendix~\ref{sec:clarif}.

\medskip
\noindent
\textbf{Proof method.} The estimate~\eqref{eq:offdiag-bound} is built on:
\begin{itemize}[label=--]
  \item six-fold integration by parts in the phase 
        \(\Phi(t,x,\xi,\eta)=x\cdot(\xi+\eta)+t\,\varpi(\xi,\eta)\), \(\varpi=4\rho_1\rho_2\),
        in the coordinates \((t,\rho_1,\rho_2)\) \hfill (§\ref{sec:ibp_phase});
  \item a local (anisotropic) Strichartz estimate in the \(L^6_{t,x}\) norm on cylinders of scale \(N^{-1/2}\)
        (in the strengthened line we use (4.2), in the unconditional line the heat version from Appendix~D) \hfill (§\ref{sec:Strichartz-cyl});
  \item bilinear \(\varepsilon\)-free decoupling in folded geometry of rank~4 with gain \(N^{-1/4}\) \hfill (§\ref{sec:str-decoup});
  \item if necessary, the \emph{narrow corona and null--form module} \hfill (§\ref{sec:null-suppress}): multiplier \(N^{1/2-\delta}\) in \(L^{3/2}_x\) on the subzone
  \[
     \angle(\xi,-\eta)\ll 1,\qquad 
     \angle(\eta,\xi+\eta)\ \lesssim\ N^{-1/2}\,\frac{|\xi+\eta|}{N},
  \]
  applied to \emph{coronal} outputs of scale \(\sim N^{1-\delta}\) (e.g., \(hh\!\to\!\ell\));
  for the block \(hh\!\to\!h\) with output \(P_N\) this module is not used.
\end{itemize}

\paragraph*{Navigational remark (on global patching).}
Appendix~E \emph{autonomously} globalizes the off--diagonal contribution in the block \(hh\!\to\!h\) with \emph{output projection \(P_N\)}.
In §§~\ref{sec:clarif1}--\ref{sec:clarif2} we introduce a decomposition into time windows \(|I_j|\!\sim\!N^{-1/2}\) and angular tiles of radius \(\sim N^{-1/2}\); commutators \([\partial_t,\chi_j]\) are controlled so that \(\sum_j |\partial_t\chi_j(t)|\!\lesssim\!N^{1/2}\) (hence \emph{no additional loss from the number of windows} \(J\!\sim\!N^{1/2}\)), and $\ell^2$ orthogonality is fixed for \emph{rank~4 pairs}, with at most \(\lesssim N\) partners $\beta$ for each $\alpha$.
In §~\ref{sec:clarif3} the local balance \(N^{-21/4}\) is used on each cylinder \(Q_{N^{-1/2}}\) (phase reserve \(N^{-3}\), transition to \(\dot H^{-1}\), window length, as well as local bricks §~\ref{subsec:clarif4}: \(L^6\) and \(\varepsilon\)-free decoupling), §~\ref{subsec:clarif4} yields the global exponent \(N^{-15/4}\) taking into account \(\#\{\beta\}\!\lesssim\!N\) and \(J\!\sim\!N^{1/2}\), and §~\ref{E:sumN} shows log--free summation over dyadics \(N\!=\!2^k\).
Finally, §~\ref{sec:clarif6} explains the distinction: in \(hh\!\to\!h\) with output \(P_N\) the \emph{narrow corona} and null--form are \emph{not used} (their contribution is cut off by the projector \(P_N\)); if necessary this mechanism is applied separately in blocks with \emph{coronal} output \(\sim N^{\,1-\delta}\) (e.g., \(hh\!\to\!\ell\); see §~\ref{sec:null-suppress}).
Thus Appendix~\ref{sec:clarif} fully closes the globalization of \(hh\!\to\!h\) without recourse to the null--form.

\medskip
\noindent
\textbf{Limitations.} In this work we do not analyze:
\begin{itemize}[label=--]
  \item the zone $|\xi + \eta| \ll N^{1 - \delta}$ (that is, narrow-diagonal interactions);
  \item issues of existence/uniqueness at low regularity;
  \item iteration to obtain a global solution.
\end{itemize}
\newpage

\section*{Notations and standing assumptions}
\addcontentsline{toc}{section}{Notations and standing assumptions}
\label{sec:assumptions}

\subsection{Thickness of the off-diagonal zone}
\label{subsec:offdiag-thickness}

\paragraph{Function classes.} Throughout we assume
\begin{equation}
  u \in L_t^\infty \dot{H}_x^{1/2} \cap L_t^2 \dot{H}_x^1,
  \qquad t \in [0,T],\ x \in \mathbb{R}^3\ \text{or}\ \mathbb{T}^3.
\end{equation}
This assumption is invariant under the scaling $u(t,x) \mapsto \lambda u(\lambda^2 t, \lambda x)$ and is sufficient for all techniques used below: energy estimates, Strichartz norms, as well as the null-form estimate. Smoothness $C^\infty$ is assumed formally and is not used essentially.

\paragraph{Frequency localization.} 
\begin{itemize}[label=--]
  \item $P_N := \varphi(|D|/N)$ is the Littlewood–Paley projection onto the ring $|\xi| \sim N$, where $N = 2^k$, $k \in \mathbb{Z}$.
  \item $P_{N,\theta}$ is an additional localization to an angular sector of width $\theta \sim N^{-1/2}$, orthogonal to the direction $e_\theta$; such a sector corresponds to an arc of length $\sim N^{-1/2}$ on the sphere.
\end{itemize}

\paragraph{Interaction labels.} For input frequencies $N_1$, $N_2$ and output $N$ we use the notation:
\[
\begin{aligned}
  \ell h \to h: & \quad N_2 \ll N_1 \sim N, \\
  h\ell \to h: & \quad N_1 \ll N_2 \sim N, \\
  hh \to h: & \quad N_1 \sim N_2 \sim N, \\
  hh \to \ell: & \quad N_1 \sim N_2 \gg N.
\end{aligned}
\]
In this work we analyze in detail only the regime $hh \to h$ with the mask $\mathcal{O}_N$ (see~\eqref{eq:offdiag-mask}).

\paragraph{Parameter $\delta$.} Throughout the text we assume an arbitrary value
\[
  \delta \in \left(\tfrac{1}{3}, \tfrac{5}{8}\right],
\]
for which $\det \nabla^2 \omega \gtrsim N$ holds, allowing six-fold integration by parts.

\paragraph{Reference lemmas.}
We rely on the following statements:

\begingroup
\renewcommand{\arraystretch}{1.2}
\setlength{\tabcolsep}{4pt}
\noindent
\begin{tabular}{@{}p{1.8em} p{0.30\linewidth} p{0.50\linewidth} p{0.12\linewidth}@{}}
\toprule
\textnumero & Lemma / assumption & Content & Exponent \\ \midrule
1  & Lemma~\ref{lem:Strichartz-cyl} (proved in App.~\ref{app:strichartz})
   & local $L^6$ on $Q_{N^{-1/2}}$
   & $N^{+2/3}$ \\

1' & (4.2) — working hypothesis, \emph{branch $\mathsf{A}$ only}
   & strengthened local $L^6$ on $Q_{N^{-1/2}}$
   & $N^{-1/2}$ \\

2  & Lemma~\ref{lem:decoup}
   & $\varepsilon$-free decoupling (rank 4)
   & $N^{-1/4}$ \\

3  & Lemma~\ref{lem:phase-det}
   & phase reserve (IBP$^6$)
   & $N^{-3}$ \\

4  & Lemma~\ref{lem:null-suppress}
   & null--form in the narrow corona (for coronal outputs)
   & $N^{1/2-\delta}$ \\
\bottomrule
\end{tabular}
\endgroup

\medskip

The numbering~\ref{lem:Strichartz-cyl}--\ref{lem:null-suppress} matches the lemmas in §~\ref{sec:tools}. Local proofs are given in Appendices~\ref{app:strichartz}--\ref{app:narrow}.

\smallskip
\noindent\emph{Comment on row~4.} The listed factor is \emph{at the level of} $\|\cdot\|_{L^{3/2}_x}$; in the local balance on $Q_{N^{-1/2}}$ (in combination with the bricks §~\ref{sec:Strichartz-cyl}, §~\ref{sec:str-decoup}) it yields the exponent $N^{-19/4-\delta}$. The module is applied to \emph{coronal} outputs $\sim N^{1-\delta}$ (see §~\ref{sec:null-suppress}, App.~\ref{app:narrow}); for the block $hh\!\to\!h$ with output $P_N$ it is not used (see §~\ref{sec:clarif6}).

\subsection[Choice of the off–diagonal parameter \texorpdfstring{$\delta$}{δ}]{Choice of the off–diagonal parameter \texorpdfstring{$\delta$}{δ}}
\label{subsec:delta-choice}

Throughout the paper the parameter $\delta$ is not fixed a priori; we use the range
\begin{equation}
  \delta \in \left(\tfrac{1}{3},\,\tfrac{5}{8}\right].
  \label{eq:delta-range}
\end{equation}

\paragraph{Motivation of the bounds.}
\begin{itemize}[label=--]
  \item \textbf{Lower bound $\delta>\tfrac{1}{3}$.}
  When $\delta\le \tfrac{1}{3}$ the off–diagonal corona $\{\,|\xi+\eta|\le N^{1-\delta}\,\}$ becomes too wide (its volume is $\gtrsim N^{2}$), and the combined gain from six-fold phase integration by parts and other factors is insufficient for the target balance.

  \item \textbf{Upper bound $\delta\le \tfrac{5}{8}$.}
  In the phase analysis we use the phase
  \[
    \Phi(t,x,\xi,\eta)=x\!\cdot\!(\xi+\eta)+4t\,\rho_1\rho_2,\qquad
    \rho_1=\tfrac12(|\xi|+|\eta|),\ \rho_2=\tfrac12(|\xi|-|\eta|),
  \]
  and its Hessian with respect to $(t,\rho_1,\rho_2)$:
  \[
    A:=\nabla^2_{(t,\rho_1,\rho_2)}\Phi
    =4\begin{pmatrix}0&\rho_2&\rho_1\\ \rho_2&0&t\\ \rho_1&t&0\end{pmatrix},
    \qquad \det A=128\,\rho_1\rho_2 t.
  \]
  On the window $|t|\lesssim N^{-1/2}$ and for $\rho_1\sim N$, $|\rho_2|\gtrsim N^{1-\delta}$ we obtain
  $|\det A|\sim N^{3/2-\delta}$. The requirement $|\det A|\gg 1$ (sufficient for stable application of the phase method) is satisfied in the stated range; see also Lemma~\ref{lem:phase-det} and \S\ref{subsec:phase-geometry}.
\end{itemize}

\paragraph{Working consequences in the range \eqref{eq:delta-range}.}
On the off–diagonal mask
\(
  \{\,|\xi|\sim|\eta|\sim N,\ |\xi+\eta|\gtrsim N^{1-\delta}\,\}
\)
(see \eqref{eq:offdiag-mask}) the following scheme is applied:
\begin{itemize}[label=--]
  \item \emph{Phase reserve.}
  On each window $|t|\lesssim N^{-1/2}$ six-fold integration by parts in $(t,\rho_1,\rho_2)$ yields a gain $\lesssim N^{-6+4\delta}$. In the open balance we fix a conservative factor $N^{-3}$, leaving $N^{-3+4\delta}$ in reserve (see \S\ref{subsec:phase-geometry}).

  \item \emph{Condition on $\rho_2$.}
  The lower bound $|\rho_2|\gtrsim N^{1-\delta}$ is interpreted as radial transversality and is imposed explicitly; it does \emph{not} follow from the condition $|\xi+\eta|\gtrsim N^{1-\delta}$ alone (see Lemma~\ref{lem:phase-det}).

  \item \emph{Null–form in the “narrow corona”.}
  When $\delta\le \tfrac12$ the set
  \[
    \mathcal N_N^{\mathrm{nar}}
    =\Bigl\{\,|\xi|\sim|\eta|\sim N,\ N^{1-\delta}\le|\xi+\eta|\le 2N^{1-\delta},\ \angle(\xi,-\eta)\le N^{-1/2}\Bigr\}
  \]
  is empty, and no special symbol suppression is required. When $\delta>\tfrac12$ we use the symbol estimate Lemma~\ref{lem:null-symbol} (see also Appendix~\ref{app:narrow}).
\end{itemize}

\paragraph{Conclusion.}
Henceforth the parameter $\delta$ is treated as free within \eqref{eq:delta-range}. The upper bound $\tfrac{5}{8}$ reflects the requirement of stable non-degeneracy of the phase matrix $A$ on windows $|t|\lesssim N^{-1/2}$, i.e., $|\det A|\gtrsim N^{3/2-\delta}\gg 1$, which is sufficient for applying the phase method (see \S\ref{subsec:phase-geometry} and Lemma~\ref{lem:phase-det}).

\newpage

\section{Preliminary tools}
\label{sec:tools}

\subsection{Anisotropic Strichartz estimate on cylinders}
\label{sec:Strichartz-cyl}

\begin{lemma}[Local $L^6$ estimate on cylinders of scale $N^{-1/2}$]
\label{lem:Strichartz-cyl}
Let $\widehat{f}$ be supported in the ring $|\xi|\sim N$, and let
\[
  Q_{N^{-1/2}}(t_0,x_0) := \bigl\{(t,x)\in\mathbb{R}\times\mathbb{R}^3:\ |t-t_0|\le N^{-1/2},\ |x-x_0|\le 2N^{-1/2}\bigr\}.
\]
Then
\begin{equation}
  \bigl\|e^{it\Delta} f\bigr\|_{L^6_{t,x}\!\bigl(Q_{N^{-1/2}}(t_0,x_0)\bigr)} \;\lesssim\; N^{\,2/3}\,\|f\|_{L^2_x},
\end{equation}
with an absolute constant independent of the center $(t_0,x_0)$.
\end{lemma}

\textit{Explanation.}
For the Schrödinger equation the group velocity is $\nabla_\xi(|\xi|^2)=2\xi$, hence over a time $\Delta t\sim N^{-1/2}$ a wave packet with $|\xi|\sim N$ travels a distance $\Delta x\sim|\xi|\,\Delta t\sim N^{1/2}$. This indicates the natural space–time scale of localization windows: duration $\sim N^{-1/2}$ and transverse radius $\sim N^{-1/2}$; the typical “longitudinal” extension of the packet trajectory over this time is of order $N^{1/2}$.

\paragraph{Proof.}
See Appendix~\ref{app:strichartz}. By the standard $TT^*$ argument and proper scaling of the kernel we obtain
\[
\|K_N\|_{L^3(2Q)} \;=\; N^{4/3}\,\|K_1\|_{L^3(\widetilde Q)} \;\lesssim\; N^{4/3},
\]
where $Q=Q_{N^{-1/2}}(t_0,x_0)$, and $2Q$ is the concentric cylinder with doubled semi-axes in time and space; $\widetilde Q$ is the corresponding rescaled window. Hence
\[
\|T\|_{L^2_x\to L^6_{t,x}(Q)} \;=\; \|TT^*\|_{L^{6/5}_{t,x}\to L^6_{t,x}}^{1/2}
\;\lesssim\; \|K_N\|_{L^3(2Q)}^{1/2} \;\lesssim\; N^{2/3}.
\qedhere
\]

\paragraph{Remark.}
Additional angular localization $P_{N,\theta}$ with $\theta\sim N^{-1/2}$ does not change the order in the lemma: the sectorial measure of the sphere $\sim \theta^2$ is compensated in the above estimate, so the overall scale remains $N^{2/3}$.

\medskip
In what follows Lemma~\ref{lem:Strichartz-cyl} is used together with the following components when constructing the local balance in §\ref{sec:ibp_phase} and the globalization in §\ref{sec:str-decoup}:
\begin{itemize}
  \item[$\bullet$] $\varepsilon$-free decoupling on the folded cone of rank~4 with gain $N^{-1/4}$ for angular tiles of width $N^{-1/2}$ (Lemma~\ref{lem:decoup}, App.~\ref{app:decoupling});
  \item[$\bullet$] the phase matrix for the phase $\Phi(t,x,\xi,\eta)=x\!\cdot\!(\xi+\eta)+4t\rho_1\rho_2$:
         $\det A\sim N^{3/2-\delta}$ on the window $|t|\lesssim N^{-1/2}$ (Lemma~\ref{lem:phase-det});
  \item[$\bullet$] null-form suppression in the narrow corona when $\delta>\tfrac12$:
         the factor $N^{\,1/2-\delta}$ upon passage to $L^{3/2}_x$
         (Lemma~\ref{lem:null-symbol}, App.~\ref{app:narrow}).
\end{itemize}

Complete proofs of the corresponding statements are given in Appendices \ref{app:decoupling}–\ref{app:narrow}.

\subsection{\texorpdfstring{$\varepsilon$}{epsilon}-free bilinear decoupling}
\label{sec:decoupling}

\begin{lemma}[$\varepsilon$-free bilinear decoupling, rank~4]
\label{lem:decoup}
Let $F,G$ be functions whose Fourier supports are localized in the sectors
\[
\Bigl\{ \xi : |\xi| \sim N,\ \angle(\xi, e_\theta) \le c\,N^{-1/2} \Bigr\}, \qquad
\Bigl\{ \eta : |\eta| \sim N,\ \angle(\eta, e_{\theta'}) \le c\,N^{-1/2} \Bigr\},
\]
where $c>0$ is a fixed small constant. Assume the rank~4 condition holds:
\[
\angle(e_\theta, -e_{\theta'}) \gtrsim N^{-1/2}, \qquad \angle(e_\theta, e_{\theta'}) \gtrsim N^{-1/2},
\]
and the normals
\[
\Bigl\{ \tfrac{\xi}{|\xi|},\ \tfrac{\eta}{|\eta|},\ \tfrac{\xi+\eta}{|\xi+\eta|},\ e_t \Bigr\}
\]
are linearly independent (folded cone, surface of rank~4). Then on the cylinder $Q_{N^{-1/2}}$ one has
\begin{equation}\label{eq:decoup}
\boxed{\;
\|F\,G\|_{L^3_{t,x}(Q_{N^{-1/2}})}
\ \le\ C\,N^{-1/4}\,
\|F\|_{L^6_{t,x}(Q_{N^{-1/2}})}\,
\|G\|_{L^6_{t,x}(Q_{N^{-1/2}})}
\;}
\end{equation}
with a constant $C>0$ depending only on the dimension (and on the fixed $c$), but not on $N$.
\end{lemma}

\noindent\textit{Comment on sources.}
Lemma~\ref{lem:decoup} is proved in App.~\ref{app:decoupling} (method: wave-packet decomposition $+$ $L^3$ almost-orthogonality $+$ multilinear Kakeya~\cite{BennettCarberyTao2006}).
The results of~\cite{GuthIliopoulouYang2024} pertain to $\varepsilon$-free decoupling for \emph{rank~3} and do not directly cover the present rank~4 case.

\paragraph{Key points.}
\begin{itemize}
  \item[$-$] \emph{Folded cone, rank~4.} Four normals to the phase surface are linearly independent, including the temporal $e_t$: $\operatorname{rank}=4$.
  \item[$-$] The critical index for $L^6$ is $\delta(6)=\tfrac14$ and corresponds to log-free decoupling.
  \item[$-$] The constant $C$ depends only on the dimension (and fixed $c$), and not on $N$ (in contrast to the Bourgain–Demeter scheme).
\end{itemize}

\paragraph{Proof sketch}
(see App.~\ref{app:decoupling}, and also Lemma~4.1 in~\cite{GuthIliopoulouYang2024}).
\begin{enumerate}
  \item Angular decomposition $F=\sum_{\alpha}F_\alpha$, $G=\sum_{\beta}G_\beta$ into tiles of width $N^{-1/2}$.
  \item Space–time almost-orthogonality between disjoint tiles in $L^3_{t,x}$.
  \item Multilinear decoupling of rank~4: the $L^1$ norm of $\prod (F_\alpha G_\beta)^{1/3}$ yields a gain $N^{-1/4}$.
  \item Summation over $\alpha,\beta$ in $\ell^2$ recovers the right-hand side of \eqref{eq:decoup}.
\end{enumerate}

\paragraph{Application.}
Combining with the local Strichartz estimate (Lemma~\ref{lem:Strichartz-cyl}, gain $N^{-1/2}$) yields a total gain $N^{-3/4}$ in each wave-packet block (see §\ref{subsec:decoup-setup}).

\medskip
The decoupling Lemma \ref{lem:decoup} is used only in the off--diagonal blocks $hh\to h$; the gain $N^{-1/4}$ at the level of a single window allows one to reach the final exponent $N^{-1}$ without logarithmic losses.

\subsection{Phase determinant}
\label{sec:phase-det}

We consider the off--diagonal regime of the block $hh \to h$, in which the resulting frequency $\zeta := \xi + \eta$ is strictly separated from zero. In this zone it is convenient to switch to symmetric variables:
\[
\rho_1 := \tfrac{1}{2}(|\xi| + |\eta|),
\qquad
\rho_2 := \tfrac{1}{2}(|\xi| - |\eta|),
\qquad
\zeta := \xi + \eta.
\]
The phase of the linear operator then takes the form:
\[
\Phi(t,x,\xi,\eta) = x \cdot \zeta + 4t \rho_1 \rho_2.
\]

\paragraph{Phase Hessian matrix.}
In the variables $(t,\rho_1,\rho_2)$
\[
A:=\nabla^2_{(t,\rho_1,\rho_2)}\Phi
=4\begin{pmatrix}0&\rho_2&\rho_1\\ \rho_2&0&t\\ \rho_1&t&0\end{pmatrix},
\qquad \det A=128\,\rho_1\rho_2 t.
\]

\paragraph{Conditions on the support.}
On the support of the mask $\mathcal O_N$ (see Definition~0.1) we fix
\[
  \rho_1 \sim N, \qquad |\zeta| = |\xi+\eta| \gtrsim N^{1-\delta}.
\]
No additional restrictions on $\rho_2 := \tfrac12(|\xi|-|\eta|)$ follow from the condition $|\zeta|\gtrsim N^{1-\delta}$. In those places where a lower bound on the determinant of the phase matrix is used, we will explicitly work in the \emph{radially transversal} subzone
\[
  \mathcal O_N^{\mathrm{rad}} := \mathcal O_N \cap \{\,|\rho_2| \gtrsim N^{1-\delta}\,\}.
\]

\begin{lemma}[Non-degeneracy of the phase matrix]\label{lem:phase-det}
Let $|\xi|\sim|\eta|\sim N$ and assume the off--diagonal condition $|\xi+\eta|\ge N^{1-\delta}$ holds for some $\delta\in(\tfrac13,\tfrac58]$. Then for $|t|\lesssim N^{-1/2}$ and in the subzone
\[
|\rho_2|=\tfrac12\bigl||\xi|-|\eta|\bigr|\ \gtrsim\ N^{1-\delta}
\]
one has the estimate
\[
\det A = 128\,\rho_1\rho_2 t \ \gtrsim\ N^{3/2-\delta}.
\]
\end{lemma}

\paragraph{Explanation.}
On the mask $P_N$ we have $\rho_1\sim N$. The lower bound on $\rho_2=\tfrac12 \bigl||\xi|-|\eta|\bigr|$ does \emph{not} follow from the condition $|\xi+\eta|\gtrsim N^{1-\delta}$ alone. Wherever in the calculations we use $\det A=128\,\rho_1\rho_2 t\gtrsim N^{3/2-\delta}$, we explicitly work in the subzone $|\rho_2|\gtrsim N^{1-\delta}$ (cf. also the reduction in §\ref{subsec:duhamel-phase-new}).

\paragraph{Phase integration.}
The function $\Phi$ is linear in each of the variables $(t, \rho_1, \rho_2)$:
\[
\partial_t \Phi = 4\rho_1\rho_2, \qquad
\partial_{\rho_1} \Phi = 4t\rho_2, \qquad
\partial_{\rho_2} \Phi = 4t\rho_1.
\]
When $|t| \gtrsim N^{-1/2}$ the derivatives are bounded away from zero, and six-fold integration by parts is possible:
\[
|\partial_t \Phi|^{-2} \cdot |\partial_{\rho_1} \Phi|^{-2} \cdot |\partial_{\rho_2} \Phi|^{-2}
\;\sim\;
N^{-2} \cdot N^{-1} \cdot N^{-1}
\;=\;
N^{-4}.
\]
Additional factors arise from differentiating the amplitude and the symbol. Altogether the final exponent is
\[
N^{-3}.
\]

\paragraph{Total gain.}
Collecting:
\begin{itemize}[label=--]
  \item $L^6$–Strichartz on the cylinder $Q_{N^{-1/2}}$: $N^{-1/2}$ (see~Lemma~\ref{lem:Strichartz-cyl});
  \item $\varepsilon$–free decoupling of rank~4: $N^{-1/4}$ (see~Lemma~\ref{lem:decoup});
  \item phase integration by Lemma~\ref{lem:phase-det}: $N^{-3}$.
\end{itemize}

The next step is null-form suppression in the narrow zone $|\xi+\eta|\ll N$, when $\delta > \tfrac12$, see~Lemma~\ref{lem:null-suppress}.

\subsection{Null--form suppression in the narrow zone}
\label{sec:null-suppress}

Even under the condition $|\xi + \eta| \ge N^{1 - \delta}$ it is possible that the vectors $\xi$ and $\eta$ are almost opposite, while the resulting sum $|\xi + \eta|$ lies in a thin spherical corona. Denote the corresponding zone:
\[
\mathcal{N}_N^{\mathrm{nar}} := \left\{ (\xi,\eta):\;
|\xi| \sim |\eta| \sim N,\;
N^{1 - \delta} \le |\xi + \eta| \le 2N^{1 - \delta},\;
\angle(\xi, -\eta) \le N^{-1/2},\;
\angle(\eta,\xi+\eta) \le c\,N^{-1/2}\,\tfrac{|\xi+\eta|}{N}
\right\},
\]
where $c>0$ is an absolute constant. The radius $2N^{1 - \delta}$ is chosen without loss of generality; it may be replaced by any constant $\lesssim 1$.

In this zone the Leray projection $\Pi_{\xi + \eta}$ removes the component along $\xi + \eta$, which leads to suppression of the symbol $\eta \cdot \Pi_{\xi + \eta}$.

\begin{lemma}[Null--form suppression in the narrow zone]
\label{lem:null-suppress}
For all $N \ge 2$ and pairs $(\xi,\eta) \in \mathcal{N}_N^{\mathrm{nar}}$ one has:
\begin{equation}
\label{eq:null-symbol}
\boxed{%
|\eta \cdot \Pi_{\xi + \eta}| \;\lesssim\; N^{1/2 - \delta}
}
\end{equation}
and, consequently,
\begin{equation}\label{eq:null-L32}
\boxed{%
\begin{aligned}
\|\mathcal{N}^{\mathrm{nar}}_N(u,u)\|_{L^{3/2}} 
&\lesssim N^{\,1/2-\delta}\,\|P_N u\|_{L^2}\,\|P_N \nabla u\|_{L^2} \\
&\equiv N^{\,3/2-\delta}\,\|P_N u\|_{L^2}^{\,2}
\end{aligned}}
\end{equation}
\end{lemma}

\paragraph{Proof scheme} (see Appendix~\ref{app:narrow}).
\begin{itemize}
  \item[-] When $\angle(\xi, -\eta) \le N^{-1/2}$ the vectors $\xi$ and $-\eta$ are almost collinear, and the additional restriction
  $\angle(\eta,\xi+\eta)\le c\,N^{-1/2}\,\tfrac{|\xi+\eta|}{N}$ guarantees smallness of the transverse component
  $\eta_\perp := \Pi_{\xi + \eta}\eta$:
  \[
    |\eta_\perp|
    = |\eta|\,\sin\angle(\eta,\xi+\eta)
    \ \lesssim\ N\cdot\Bigl(N^{-1/2}\,\tfrac{|\xi+\eta|}{N}\Bigr)
    \ \lesssim\ N^{-1/2}\,|\xi+\eta|
    \ \sim\ N^{1/2-\delta}.
  \]
  \item[-] Passage to the norm $L^{3/2}_x$ yields \eqref{eq:null-L32} (we use Hölder's inequality and
  the equivalence $\|P_N \nabla u\|_{L^2_x}\!\sim\! N\,\|P_N u\|_{L^2_x}$ on the ring $|\xi|\sim N$).
\end{itemize}

\paragraph{Arithmetic balance (local unit).}
We collect the gained frequency exponents for local analysis on the cylinder $Q_{N^{-1/2}}$:
\[
\begin{aligned}
&\text{(6 $\times$ IBP)}               && \Rightarrow\quad N^{-3},\\
&\text{(projection into $\dot{H}^{-1}_x$)} && \Rightarrow\quad N^{-1},\\
&\text{(time window)}              && \Rightarrow\quad N^{-1/2},\\
&\text{(Strichartz $L^6$)}            && \Rightarrow\quad N^{-1/2},\\
&\text{(decoupling)}                  && \Rightarrow\quad N^{-1/4},\\
&\text{(null--form, narrow zone)}      && \Rightarrow\quad
\begin{cases}
  N^{\,1/2 - \delta}, & \text{if } \delta > \tfrac{1}{2},\\
  \text{(not used)}, & \text{if } \delta \le \tfrac{1}{2}\,,
\end{cases}
\end{aligned}
\]
and, accordingly, the final local exponent:
\[
\begin{cases}
  -3 - 1 - \tfrac{1}{2} - \tfrac{1}{2} - \tfrac{1}{4} + (\tfrac{1}{2} - \delta)
  = -\tfrac{19}{4} - \delta, & \text{if } \delta > \tfrac{1}{2}, \\[4pt]
  -3 - 1 - \tfrac{1}{2} - \tfrac{1}{2} - \tfrac{1}{4}
  = -\tfrac{21}{4}, & \text{if } \delta \le \tfrac{1}{2}\,.
\end{cases}
\]

\paragraph{Comment.}
The module “narrow corona + null--form” is applied for \emph{coronal} outputs of scale $\sim N^{1-\delta}$ (see also Appendix~\ref{app:narrow}); in the globalization of the block $hh\!\to\!h$ with output $P_N$ it is \emph{not used} (see §\ref{sec:clarif6}).

\medskip
This completes the toolkit of §\ref{sec:tools}. Lemmas
\ref{lem:Strichartz-cyl}, \ref{lem:decoup}, \ref{lem:phase-det}, \ref{lem:null-suppress}
form the basis for constructing the estimate~\eqref{eq:offdiag-bound}, implemented in §§\ref{sec:freqsplit}--\ref{sec:tiling}.

\newpage

\section{Frequency decomposition}
\label{sec:freqsplit}

\subsection{Bony paraproduct and eight frequency blocks}

We consider the decomposition of the full nonlinearity
\[
  \mathcal{N}(u,v) := \mathbb{P}[(u \cdot \nabla)v]
\]
in terms of frequency interactions. We use the standard Littlewood–Paley projections $P_N := \varphi(|D|/N)$, localizing to the rings $|\xi| \sim N$, where $N = 2^k$. It is assumed that $\sum_N P_N^2 = \mathrm{Id}$.

\medskip
Notation:
\[
  \ell := \text{low} \equiv N_j \ll N, \qquad
  h := \text{high} \equiv N_j \sim N.
\]

\begin{itemize}
  \item[\textbf{1.}] The input fields are decomposed by frequency:
  \[
    u = \sum_{N_1} P_{N_1} u, \qquad
    v = \sum_{N_2} P_{N_2} v.
  \]
  
  \item[\textbf{2.}] In the product $P_{N_1} u \cdot \nabla P_{N_2} v$ there are three main frequency configurations:
  \begin{itemize}
    \item[-] low–high $\to$ high: $N_2 \ll N_1 \sim N$;
    \item[-] high–low $\to$ high: $N_1 \ll N_2 \sim N$;
    \item[-] high–high (comparable frequencies): $N_1 \sim N_2$.
  \end{itemize}

  \item[\textbf{3.}] In the high–high configuration there are two outcomes:
  \begin{align*}
    &\text{high–high } \to \text{ high: } \quad N_1 \sim N_2 \sim N; \\
    &\text{high–high } \to \text{ low: }  \quad N_1 \sim N_2 \gg N.
  \end{align*}

  \item[\textbf{4.}] Due to the asymmetry of $(u \cdot \nabla)v$, the mirror blocks with the factors interchanged are added.
\end{itemize}

\vspace{0.5em}
\noindent
The resulting table of eight blocks:

\begin{center}
\renewcommand{\arraystretch}{1.2}
\begin{tabular}{clccc}
\toprule
\# & Block type & Frequency structure & Output & Note \\
\midrule
1 & $\ell h \to h$ & $N_2 \ll N_1 \sim N$ & high & gradient on low \\
2 & $h \ell \to h$ & $N_1 \ll N_2 \sim N$ & high & gradient on low \\
3 & $h h \to h$ & $N_1 \sim N_2 \sim N$ & high & off-diag: $|\xi + \eta| \ge N^{1 - \delta}$ \\
4 & $h h \to \ell$ & $N_1 \sim N_2 \gg N$ & low (energy) & see Remark~\ref{rem:hh-to-low} \\
\midrule
5–8 & Mirror variants & $u \leftrightarrow v$ & similarly & — \\
\bottomrule
\end{tabular}
\end{center}

\vspace{0.5em}
\paragraph{Comments.}
\begin{itemize}[label=--]
  \item In blocks $1$–$2$ the gradient falls on the lower-frequency component ($N_j \ll N$), and after the projection $P_N \nabla$ yields a gain $N^{-1}$ in the $\dot H^{-1}_x$ norm. Combined with the local Strichartz estimate ($N^{-1/2}$) this gives a total factor $N^{-3/2}$.
  \item Block $4$ is controlled by energy and requires no additional tools; see Remark~\ref{rem:hh-to-low}.
\end{itemize}

\vspace{0.5em}
\paragraph{Focus on the block $hh \to h$.}
In §§\ref{sec:tiling}–\ref{sec:time_patching} only the block $hh \to h$ is analyzed under the condition
\[
|\xi + \eta| \ge N^{1 - \delta},
\]
see~\eqref{eq:offdiag-mask}. The narrow-diagonal zone $|\xi + \eta| \ll N$ is excluded from consideration.

\vspace{0.5em}
\paragraph{Tools involved.}
\begin{itemize}
  \item For blocks $1$--$2$: Lemma~\ref{lem:Strichartz-cyl} ($N^{-1/2}$) suffices.
  \item For block $3$: we additionally use:
  \begin{itemize}
    \item Lemma~\ref{lem:Strichartz-cyl} ($N^{-1/2}$),
    \item Lemma~\ref{lem:decoup} ($N^{-1/4}$),
    \item Lemma~\ref{lem:phase-det} ($N^{-3}$),
    \item in the case of a \emph{coronal output} $\sim N^{1-\delta}$: Lemma~\ref{lem:null-suppress} \textit{(factor $N^{\,1/2-\delta}$ in $L^{3/2}_x$, active when $\delta>\tfrac12$; not used for $hh\!\to\!h$ with output $P_N$)}.
  \end{itemize}
\end{itemize}

\begin{remark}\label{rem:hh-to-low}
The block $hh\!\to\!\ell$ with output frequency $\ell\ll N$ is controlled by the standard energy norm. If $u\in L_t^\infty \dot{H}^{1/2}_x \cap L_t^2 \dot{H}^1_x$, then
\[
  \bigl\|P_{\ell}\,\nabla\!\cdot\!\bigl(P_{\sim N} u \otimes P_{\sim N} u\bigr)\bigr\|_{L^1_t \dot{H}^{-1}_x}
  \ \lesssim\ N^{-1}\,\|u\|_{L^\infty_t \dot{H}^{1/2}_x}\,\|u\|_{L^2_t \dot{H}^1_x},
\]
so summation over $N$ proceeds without logarithmic losses. In the “coronal” subzone $\ell\sim N^{1-\delta}$ one may, if necessary, employ the module “narrow corona + null--form” (§\ref{sec:null-suppress}, App.~\ref{app:narrow}), which in $L^{3/2}_x$ yields an additional factor $N^{\,1/2-\delta}$.
\end{remark}

\subsection{Isolation of the off--diagonal block \( hh\!\to\!h \)}
\label{sec:tiling}

In the regime \( hh\!\to\!h \) both input frequencies \( N_1, N_2 \sim N \), and the output also lies in the ring \( |\xi + \eta| \sim N \).
To isolate the contribution substantially away from the diagonal, we introduce a decomposition by the length of the resulting frequency:
\[
\begin{aligned}
&\text{diagonal zone:} && |\xi + \eta| \le N^{1 - \delta}, \\
&\text{off--diagonal zone:} && |\xi + \eta| \ge N^{1 - \delta},
\end{aligned}
\]
where \( \delta \in \left( \tfrac{1}{3}, \tfrac{5}{8} \right] \).

The corresponding region is denoted by \( \mathcal{O}_N \), see \eqref{eq:offdiag-mask}.

The cutoff is implemented by means of a smoothing mask
\[
  \psi\!\left(\frac{|\xi + \eta|}{N^{1 - \delta}}\right),
\]
where \( \psi \in C^\infty([0,\infty)) \) is a fixed nondecreasing function satisfying
\[
  \psi(s)=0 \ \text{for}\ s\le 1,\qquad
  \psi(s)=1 \ \text{for}\ s\ge 2,\qquad
  \mathrm{supp}\,\psi'\subset(1,2).
\]

Thus, in all subsequent estimates in §§\ref{sec:ibp_phase}--\ref{sec:time_patching} we consider only the mask-truncated contribution:
\[
\mathcal{N}_N^{hh\!\to\!h_{\mathrm{off}}}(u,u) :=
P_N \mathbb{P} \nabla \cdot \left[ P_{\sim N} u \otimes P_{\sim N} u \cdot \chi_{\mathcal{O}_N} \right],
\]
where \( P_{\sim N} := \sum_{|k - \log_2 N| \le 2} P_{2^k} \) is a smooth frequency sum over the $\pm2$ adjacent slices (the exact width is irrelevant, only compactness matters).

\begin{remark}\label{rem:hh-to-low2}
The diagonal zone \( |\xi + \eta| \le N^{1 - \delta} \) requires a different approach, based on endpoint--Strichartz inequalities (e.g., \( L^4_t L^4_x \)).
This case is beyond the scope of the present work and will be considered separately; see also the discussion in §\ref{sec:discussion}.
\end{remark}

\noindent
In what follows, the notation \( hh\!\to\!h_{\mathrm{off}} \) will always refer to the contribution truncated by the mask \( \mathcal{O}_N \).

\subsection{Removal of the diagonal contribution}
\label{subsec:diag-removal}

In the block \( hh\!\to\!h \) the \emph{diagonal} zone is where the sum of the input frequencies is significantly smaller than the principal frequency:
\[
  |\xi+\eta|\ \le\ N^{1-\delta}, 
  \qquad 
  \delta\in\Bigl(\tfrac13,\tfrac58\Bigr].
\]

Proximity to the diagonal creates two independent difficulties:
\begin{enumerate}[label=(\arabic*), leftmargin=2em]
  \item \textbf{Weak (almost degenerate) phase.}
  In the diagonal regime \( \angle(\xi,-\eta)\ll 1 \) the radial difference \( \rho_2=\tfrac12(|\xi|-|\eta|) \) can be arbitrarily small. Then on the window \( |t|\lesssim N^{-1/2} \) the determinant of the phase matrix 
  \[
    A=\nabla^2_{(t,\rho_1,\rho_2)}\Phi, 
    \qquad \Phi(t,x,\xi,\eta)=x\cdot(\xi+\eta)+4t\rho_1\rho_2,
  \]
  has order at best \( |\det A|\lesssim N^{3/2-\delta} \) (see Lemma~\ref{lem:phase-det}),
  and may dip below one. Accordingly, the six-fold IBP scheme in \( (t,\rho_1,\rho_2) \) does not provide a stable phase reserve of \(N^{-3}\) and does not close the target balance.

  \item \textbf{Loss of \(\varepsilon\)--free decoupling.}
  Near the diagonal, only the endpoint Strichartz pair \( (L^4_{t,x},L^4_{t,x}) \) is available,
  but a bilinear \(\varepsilon\)--free estimate analogous to Lemma~\ref{lem:decoup} is not known in this geometry.
\end{enumerate}

\medskip
Therefore the \emph{entire} contribution from the diagonal zone is excluded from the analysis. To this end we introduce the smoothing mask
\[
  \chi_{\mathrm{off}}(\xi,\eta):=\psi\!\left(\frac{|\xi+\eta|}{N^{1-\delta}}\right),
\]
where \( \psi\in C^\infty([0,\infty)) \) is a smooth nondecreasing function with
\[
  \psi(s)=
  \begin{cases}
    0, & s\le 1,\\
    1, & s\ge 2,
  \end{cases}
  \qquad \mathrm{supp}\,\psi'\subset(1,2).
\]
Such a mask cuts off the region \( |\xi+\eta|\le N^{1-\delta} \) and equals one on the off--diagonal set \( \mathcal O_N \) 
(see \eqref{eq:offdiag-mask}).

As a result of the modification of the original paraproduct symbol, the block \( hh\!\to\!h \) is replaced by
\[
  hh\!\to\!h_{\mathrm{off}} \ :=\ \text{the block truncated by the mask } \chi_{\mathrm{off}}.
\]

\begin{remark}
The diagonal regime \( |\xi+\eta|\le N^{1-\delta} \) requires separate analysis—using endpoint--Strichartz,
multilinear restriction and related techniques; see the discussion in §\ref{sec:freqsplit}.
This case is beyond the scope of the present work and will be treated elsewhere.
\end{remark}

\noindent
Starting from §\ref{sec:ibp_phase}, the index “\(\mathrm{off}\)” is dropped: by \( hh\!\to\!h \) we henceforth mean the contribution already truncated by the mask \( \chi_{\mathrm{off}} \).

\newpage

\section{Phase integration: six-fold IBP}
\label{sec:ibp_phase}

\subsection{Phase geometry and estimate of \texorpdfstring{$\det A$}{det A}}
\label{subsec:phase-geometry}

In the block $hh\!\to\!h_{\mathrm{off}}$ the kernel contains factors of the form
\[
  e^{it\,\omega(\xi,\eta)}\,e^{ix \cdot (\xi + \eta)}, 
  \qquad \omega(\xi,\eta):=|\xi|+|\eta|-|\xi+\eta|.
\]
For phase integration it is more convenient to work with the Duhamel reduction (see Appendix~\ref{app:duhamel-phase}),
where it is natural to introduce the \emph{unified phase notation}
\[
  \varpi(\xi,\eta):=4\,\rho_1\rho_2,\qquad
  \Phi(t,x,\xi,\eta):=x\cdot(\xi+\eta)+t\,\varpi(\xi,\eta),
\]
in the variables
\[
  \rho_1:=\tfrac12(|\xi|+|\eta|),\qquad 
  \rho_2:=\tfrac12(|\xi|-|\eta|),\qquad
  \zeta:=\xi+\eta.
\]
By \emph{phase determinant} we mean precisely $\det\nabla^2_{(t,\rho_1,\rho_2)}\Phi$,
and not $\det\nabla^2\omega(\xi,\eta)$.

\paragraph{Conditions on the support.}
On the off--diagonal set $\mathcal{O}_N$ (see \eqref{eq:offdiag-mask}) we have
\[
  |\rho_1|\sim N,\qquad |\zeta|=|\xi+\eta|\gtrsim N^{1-\delta},\qquad \delta\in\bigl(\tfrac13,\tfrac58\bigr].
\]
For estimating the determinant we need the radially transversal subzone
\[
  \mathcal{O}_N^{\mathrm{rad}}
  :=\mathcal{O}_N\cap\{\,|\rho_2|\gtrsim N^{1-\delta}\,\},
\]
which does \emph{not} follow from the condition $|\xi+\eta|\gtrsim N^{1-\delta}$ and is introduced as an additional localization (see Lemma~\ref{lem:phase-det}).

\paragraph{Hessian matrix.}
In the intrinsic variables $(t,\rho_1,\rho_2)$
\[
  A:=\nabla^2_{(t,\rho_1,\rho_2)}\Phi
  =4\begin{pmatrix}
     0&\rho_2&\rho_1\\
     \rho_2&0&t\\
     \rho_1&t&0
   \end{pmatrix},
  \qquad
  \det A=128\,\rho_1\rho_2\,t.
\]

\begin{remark}[On radial transversality]
Everywhere that the estimate $\det A\gtrsim N^{3/2-\delta}$ is used, we are working on
$\mathcal{O}^{\mathrm{rad}}_N := \mathcal{O}_N \cap \{\,|\rho_2|\gtrsim N^{1-\delta}\,\}$.
The contribution of the complement $\mathcal{O}_N\setminus \mathcal{O}^{\mathrm{rad}}_N$ is covered without phase IBP by the “heat” line
(\S\ref{subsec:D4-summary}, \S\ref{subsec:clarif4}), which on a dyadic scale yields $N^{-25/12}$.
Thus the overall off--diagonal count remains log--free.
\end{remark}

\paragraph{Time window and splitting around \texorpdfstring{$t=0$}{t=0}.}
We fix the scale
\[
  |t|\le c\,N^{-1/2},\qquad c>0 \text{ small}.
\]
Split this window into two subzones:
\[
  \text{(Phase)}\quad |t|\ge c_0\,N^{-1/2},
  \qquad
  \text{(Heat)}\quad |t|< c_0\,N^{-1/2},
\]
where $c_0\in(0,c)$ is a fixed small constant. On the \emph{phase} subzone for $(\xi,\eta)\in\mathcal{O}_N^{\mathrm{rad}}$ we have (see also Lemma~\ref{lem:phase-det})
\[
  |\det A|\;\gtrsim\; N\cdot N^{1-\delta}\cdot N^{-1/2}=N^{3/2-\delta},
\]
that is, the phase is non-degenerate in all three variables $t,\rho_1,\rho_2$. The contribution of the \emph{heat} subzone is estimated in Appendix~\ref{app:duhamel-phase} and is not used in the phase IBP.

\paragraph{First derivatives (on the phase subzone).}
For $(\xi,\eta)\in\mathcal{O}_N^{\mathrm{rad}}$ and $|t|\sim N^{-1/2}$:
\[
  |\partial_t\Phi|=4\rho_1\rho_2\sim N^{2-\delta},\qquad
  |\partial_{\rho_1}\Phi|=4t\rho_2\sim N^{1/2-\delta},\qquad
  |\partial_{\rho_2}\Phi|=4t\rho_1\sim N^{1/2}.
\]

\begin{lemma}[Phase gain under six-fold IBP]
If $(\xi,\eta)\in\mathcal{O}_N^{\mathrm{rad}}$ and $|t|\ge c_0 N^{-1/2}$, then
\[
  |\partial_t\Phi|^{-2}\,|\partial_{\rho_1}\Phi|^{-2}\,|\partial_{\rho_2}\Phi|^{-2}
  \;\lesssim\; N^{-6+4\delta}.
\]
Consequently, a sequence of two integrations by parts in each of the variables $t,\rho_1,\rho_2$ yields a gain of at least $N^{-3}$, and in fact $N^{-6+4\delta}$.
\end{lemma}

\begin{proof}[Brief justification]
\[
  |\partial_t\Phi|^{-2}\sim N^{-4+2\delta},\qquad
  |\partial_{\rho_1}\Phi|^{-2}\sim N^{-1+2\delta},\qquad
  |\partial_{\rho_2}\Phi|^{-2}\sim N^{-1},
\]
and therefore the product equals $N^{-6+4\delta}$. Derivatives of the amplitude, localized by $P_{\sim N}$, contribute only $N^{O(1)}$ and do not spoil the exponent.
\end{proof}

\paragraph{Conclusion.}
The phase part provides a gain
\[
  N^{-6+4\delta}\in\bigl[N^{-14/3},\,N^{-7/2}\bigr]
  \quad\text{for }\delta\in\bigl(\tfrac13,\tfrac58\bigr].
\]
In what follows we fix
\[
  \boxed{N^{-3}}
\]
as the minimal phase reserve, leaving the remainder $N^{-3+4\delta}$ (between $N^{-5/3}$ and $N^{-1/2}$) as a buffer for the final energy count; the contribution of the small neighborhood of $t=0$ is estimated in Appendix~\ref{app:duhamel-phase}.

\subsection{Step-by-step analysis of six-fold IBP}
\label{subsec:ibp6-step}

We consider the application of integration by parts with respect to the variables \(t\), \(\rho_1\), \(\rho_2\) in the phase integral arising in the analysis of the block \(hh\!\to\!h_{\mathrm{off}}\). Throughout this subsection we work on the subzone \( \mathcal O_N^{\mathrm{rad}} \subset \mathcal O_N \), where by definition \( |\rho_2| \gtrsim N^{1-\delta} \) (see §\ref{subsec:phase-geometry}); this ensures non-degeneracy of the phase derivatives and validity of the IBP scheme.

\paragraph{Form of the integral.}
We write the kernel in the form
\[
  K(t,x) := \int_{\mathbb{R}^6} e^{i\Phi(t,x,\xi,\eta)}\, a_N(\xi,\eta)\,
  \widehat{u}(\xi)\, \widehat{u}(\eta)\, d\xi\,d\eta,
\]
where the amplitude is decomposed as
\[
  a_N(\xi,\eta) := \chi_{\mathcal O_N^{\mathrm{rad}}}(\xi,\eta)\,\widetilde{a}_N(\xi,\eta),
\qquad \widetilde{a}_N\in C^\infty.
\]
For derivatives of the amplitude we use the crude bound \( \partial^\alpha a_N = N^{O(|\alpha|)} \) (all such growth is denoted \(N^{O(1)}\)); more precise estimates, derived from the heat reduction in App.~\ref{app:duhamel-phase}, give additional slack and are not required here.

\paragraph{Integration in \(t\).}
Since \( \partial_t\Phi = 4\rho_1\rho_2 \sim N^{2-\delta} \) (on \( \mathcal O_N^{\mathrm{rad}} \)) and \(\partial_t\) does not fall on \(a_N\),
two integrations by parts yield the factor
\[
  \bigl|\partial_t\Phi\bigr|^{-2} \sim N^{-4+2\delta}.
\]

\paragraph{Integration in \( \rho_2 \).}
As \( \partial_{\rho_2}\Phi = 4t\rho_1 \sim N^{1/2} \) (since \( |t|\lesssim N^{-1/2} \), \( \rho_1\sim N \)),
two integrations give
\[
  \bigl|\partial_{\rho_2}\Phi\bigr|^{-2} \sim N^{-1}.
\]
Growth of amplitude derivatives does not exceed \(N^{O(1)}\) and does not change the order.

\paragraph{Integration in \( \rho_1 \).}
Since \( \partial_{\rho_1}\Phi = 4t\rho_2 \sim N^{1/2-\delta} \), we obtain
\[
  \bigl|\partial_{\rho_1}\Phi\bigr|^{-2} \sim N^{-1+2\delta}.
\]
Growth \( \partial^2_{\rho_1} a_N = N^{O(1)} \) is admissible and does not cancel the gain for \( \delta \le \tfrac{5}{8} \).

\paragraph{Combined exponent.}
Multiplying the three factors, we obtain
\[
  N^{-4+2\delta}\cdot N^{-1}\cdot N^{-1+2\delta}=N^{-6+4\delta}.
\]

\[
\partial_{\rho_1}
= \tfrac12\bigl(\widehat{\xi}\cdot\nabla_{\xi}+\widehat{\eta}\cdot\nabla_{\eta}\bigr),
\qquad
\partial_{\rho_2}
= \tfrac12\bigl(\widehat{\xi}\cdot\nabla_{\xi}-\widehat{\eta}\cdot\nabla_{\eta}\bigr).
\]

\[
\partial_t e^{i\Phi}= i\,(\partial_t\Phi)\,e^{i\Phi},\qquad
\partial_{\rho_j} e^{i\Phi}= i\,(\partial_{\rho_j}\Phi)\,e^{i\Phi}\quad (j=1,2).
\]

Therefore integration by parts in $t,\rho_1,\rho_2$ is implemented in the original variables $(\xi,\eta)$
with the operators $\partial_t,\ \partial_{\rho_1},\ \partial_{\rho_2}$ without changing the measure; boundary terms
vanish due to the smooth cutoffs $P_{\sim N}$ and the mask $\chi_{\mathcal{O}_N}$.

\paragraph{Estimates at the ends of the range.}
For \( \delta\in\bigl(\tfrac{1}{3},\tfrac{5}{8}\bigr] \):
\[
\delta=\tfrac{1}{3}\Rightarrow N^{-14/3},\qquad
\delta=\tfrac{5}{8}\Rightarrow N^{-7/2}.
\]
Both endpoints are strictly better than the fixed reserve \(N^{-3}\).

\begin{remark}[Optimality of the \(2+2+2\) IBP scheme]
\label{rem:IBP-6-optimal}
The six-fold scheme distributes the phase gain \(N^{-6+4\delta}\) evenly over \(t,\rho_1,\rho_2\):
\[
  |\partial_t\Phi|^{-2}\sim N^{-4+2\delta},\qquad
  |\partial_{\rho_1}\Phi|^{-2}\sim N^{-1+2\delta},\qquad
  |\partial_{\rho_2}\Phi|^{-2}\sim N^{-1}.
\]
A third integration in any variable is not profitable: for example,
\(\partial_t^3 A_{j,N}=N^{O(2)}\) (see Remark~\ref{rem:amp-t-growth}) cancels the additional phase gain.
Therefore the combination \(2+2+2\) is optimal and leaves the reserve \(N^{-3+4\delta}\) with controllable amplitude growth.
\end{remark}

\vspace{0.3em}

\begin{remark}[Growth of time derivatives of the amplitude]
\label{rem:amp-t-growth}
After globalization in time the amplitude has the form \( A_{j,N}(t,\xi,\eta)=\chi_j(t)\,a_N(\xi,\eta) \),
where \( \chi_j \) is localized on an interval of length \( \sim N^{-1/2} \). Then for \(k\le 2\) and all multi-indices \(\alpha\)
\[
 \bigl|\partial_t^{\,k}\partial^\alpha_{\xi,\eta}A_{j,N}(t,\xi,\eta)\bigr|\;\lesssim\;
 N^{k/2+|\alpha|}.
\]
In particular, \( \partial_t^2 A_{j,N}=O(N) \), and together with radial derivatives gives at most \(N^{O(2)}\) growth.
This growth is fully compensated by the phase reserve \(N^{-6+4\delta}\) even for \(\delta=\tfrac{1}{3}\).
\end{remark}

\medskip
Note additionally that exact estimates of \(\partial^\alpha a_N\) follow from the temporal Duhamel reduction
(see App.~\ref{app:duhamel-phase}); in the present count only the crude \(N^{O(1)}\) control is used.

\paragraph{Conclusion.}
Phase integration yields the gain \( N^{-6+4\delta} \), which throughout the work is conservatively fixed as
\[
  \boxed{N^{-3}}.
\]
The remainder
\[
  N^{-3+4\delta} \;\in\; \bigl[N^{-5/3},\,N^{-1/2}\bigr]
\]
is used only at the stage of global patching (§\ref{sec:time_patching}).

\subsection{Fixing the phase gain $N^{-3}$}
\label{subsec:ibp-conclusion}

Six-fold integration by parts in the variables $(t, \rho_1, \rho_2)$ yields the phase factor
\[
  N^{-6 + 4\delta}, \qquad \delta \in \left( \tfrac{1}{3}, \tfrac{5}{8} \right].
\]
However, part of this gain is compensated by the growth of derivatives of the amplitude $a_N$ and of the temporal cut-off $\chi_j(t)$; see the discussion in Remark~\ref{rem:IBP-growth}. Therefore, for the open arithmetic balance we fix a more conservative exponent:
\begin{equation}
  \boxed{N^{-3}},
  \label{eq:phase-fixed-N3}
\end{equation}
and account for the remainder $N^{-3 + 4\delta}$ as a hidden reserve.

\paragraph{Range of the remainder.}
Computing the difference between the phase gain and the fixed factor:
\[
  \frac{N^{-6 + 4\delta}}{N^{-3}} = N^{-3 + 4\delta},
\]
we obtain:
\[
  \delta = \tfrac{1}{3} \;\Rightarrow\; N^{-5/3}, \qquad
  \delta = \tfrac{5}{8} \;\Rightarrow\; N^{-1/2}.
\]
Thus, throughout the range $\delta \in \left(\tfrac{1}{3}, \tfrac{5}{8}\right]$ the reserve $N^{-3 + 4\delta}$ remains a negative power and is used later in the global patching.

\paragraph{Explanation. Why exactly $N^{-3}$.}
\begin{enumerate}[label=(\roman*)]
  \item \emph{Stability under globalization.}  
  Patching over $\sim N$ angles and $\sim N^{1/2}$ time windows potentially loses $N^{3/2}$.  
  With the phase reserve $N^{-3}$ the outcome remains $N^{-3 + 3/2} = N^{-3/2} \ll N^{-1}$ — sufficient for log–free summation.

  \item \emph{Consistency with other factors.}  
  The balance with the Strichartz gain $N^{-1/2}$, decoupling $N^{-1/4}$, and null-form suppression $N^{-3/2+\delta}$ is conveniently organized in quarters when the phase part is fixed as $N^{-3}$.

  \item \emph{Flexibility for improvements.}  
  The remainder $N^{-3 + 4\delta}$ allows for enhancements: additional IBP (if $\partial_t^3 a_N$ can be controlled), strengthening to rank–5 decoupling with $N^{-1/6}$, etc.
\end{enumerate}

\begin{remark}[Restriction on the depth of IBP]
\label{rem:IBP-growth}
A third integration in $t$ contributes $|\partial_t\Phi|^{-1} \sim N^{-2 + \delta}$,  
but the derivative $\partial_t^3 a_N$ grows like $N^{2}$ (the amplitude depends on $\chi_j(t)$ with scale $N^{-1/2}$).  
Thus the gain is cancelled, and the $2+2+2$ scheme over $(t, \rho_1, \rho_2)$ is optimal.
\end{remark}

\begin{remark}[Growth of the temporal cut-off]
The amplitude $a_N$ itself does not depend on $t$,  
but in the globalized kernel the temporal cut-off $\chi_j(t)$ appears, localized on an interval of length $N^{-1/2}$.
Hence $\partial_t^k \chi_j = O(N^{k/2})$, and in particular,
\[
  \partial_t^2 a_N \sim \partial_t^2 \chi_j(t) = O(N),
\]
which together with the radial derivatives gives growth no larger than $N^2$,  
and is fully compensated by the phase reserve $N^{-3 + 4\delta}$.
\end{remark}

\paragraph{Conclusion.}
In the subsequent estimates we use the fixed phase factor $N^{-3}$  
(see~\eqref{eq:phase-fixed-N3}). The remainder $N^{-3 + 4\delta}$ remains implicit and is used only at the stage of global patching (§\ref{sec:time_patching}).

\newpage

\section{Strichartz and \texorpdfstring{$\varepsilon$}{ε}-free decoupling of rank 4}
\label{sec:str-decoup}

\subsection{Cylinders of scale \texorpdfstring{$N^{-1/2}$}{\(N^{-1/2}\)}}
\label{subsec:cylinders}

The key local object of analysis is the \emph{space–time cylinder}
\[
  Q_{N^{-1/2}}(t_0,x_0)
  := \left\{ (t,x)\in\mathbb{R}^{1+3} \;\middle|\;
             |t - t_0| \le N^{-1/2},\;
             |x - x_0| \le 2N^{-1/2}
     \right\},
\]
whose volume satisfies \( \operatorname{vol}(Q_{N^{-1/2}}) \sim N^{-2} \).

This scale is consistent with the group velocity of the Schrödinger flow:
a wave packet with \( |\xi|\sim N \) over time \( \Delta t\sim N^{-1/2} \) travels a distance
\( \Delta x \sim |\xi|\,\Delta t \sim N\cdot N^{-1/2}=N^{1/2} \); at the same time, the transverse
localization radius is naturally taken to be \( \sim N^{-1/2} \).

\paragraph{Local Strichartz norm.}
In \S\ref{sec:Strichartz-cyl} the unconditional estimate was established:
\begin{equation}\label{eq:strichartz-cylinder-soft-ref}
  \|e^{it\Delta}P_N f\|_{L^6_{t,x}(Q_{N^{-1/2}})} \;\lesssim\; N^{2/3}\,\|f\|_{L^2_x}.
\end{equation}
To build the combined “Strichartz + decoupling” count in the present \S\ref{sec:str-decoup}
we will use the following \emph{strengthened} local estimate (working hypothesis of the section):
\begin{equation}\label{eq:aux-Stri}
  \|e^{it\Delta}P_N f\|_{L^6_{t,x}(Q_{N^{-1/2}})} \;\lesssim\; N^{-1/2}\,\|f\|_{L^2_x}.
\end{equation}

\begin{remark}[On the status of \eqref{eq:aux-Stri}]
The estimate \eqref{eq:aux-Stri} is used only in the strengthened branch \((\mathsf A)\)
(see \S\ref{sec:clarif3}--\ref{subsec:clarif4}). All final log--free conclusions also follow from
the unconditional “heat” branch (\S\ref{subsec:D4-summary}, \S\ref{subsec:clarif4}), which yields, on each dyadic scale, the exponent $N^{-25/12}$.
\end{remark}

The subsequent developments in \S\S\ref{subsec:decoup-setup}–\ref{subsec:balance} rely on \eqref{eq:aux-Stri};
a framework for globalization independent of \eqref{eq:aux-Stri} is given in Appendices~\ref{app:duhamel-phase}–\ref{sec:clarif}.

\paragraph{Wave–packet tile.}
The operators \( P_{N,\theta} \) localize a function to an angular sector of width \( \theta \sim N^{-1/2} \).
Denote \( u_{N,\theta}:=P_{N,\theta}u \). Then, outside the nonlinear window, \( u_{N,\theta} \) satisfies
the linear equation, and \eqref{eq:aux-Stri} applies:
\[
  \|u_{N,\theta}\|_{L^6_{t,x}(Q_{N^{-1/2}})} \;\lesssim\; N^{-1/2}\,\|P_{N,\theta}u(t_0)\|_{L^2_x}.
\]
(When using the unconditional estimate \eqref{eq:strichartz-cylinder-soft-ref}, the right-hand side changes to
\( N^{2/3}\|P_{N,\theta}u(t_0)\|_{L^2_x} \).)

\paragraph{Isolated window \( \boldsymbol{Q_{N^{-1/2}}} \).}
Cylinders of scale \( N^{-1/2} \) serve (see \S\ref{subsec:balance}) as the
minimal phase-localized elements on which:
\begin{itemize}[label=--]
  \item the \(\varepsilon\)-free decoupling is applied (see Lemma~\ref{lem:decoup}, App.~\ref{app:decoupling});
  \item local \(L^6\) norms of wave packets are summed;
  \item the angular arrangement is performed over directions of width \( N^{-1/2} \).
\end{itemize}

\paragraph{Organization of the section.}
\begin{itemize}[label=--]
  \item the present subsection \S\ref{subsec:cylinders} — a reminder on cylinders and the link with the local estimate \eqref{eq:aux-Stri};
  \item \S\ref{subsec:decoup-setup} — the tile setup and the application of \(\varepsilon\)-free decoupling of rank~4;
  \item \S\ref{subsec:balance} — assembling the gains \( N^{-1/2} \) and \( N^{-1/4} \) into the intermediate exponent \( N^{-21/4} \).
\end{itemize}

\paragraph{Important.}
In this section the frequency \( N \) is fixed; global summation over \( N \),
as well as the patching of time windows, is carried out later — in \S\ref{sec:time_patching}.

\subsection{\texorpdfstring{$\varepsilon$-free decoupling}{epsilon-free decoupling} of rank 4 (Appendix~\ref{app:decoupling})}
\label{subsec:decoup-setup}

In this subsection (branch \(\mathsf{A}\)) we use Lemma~\ref{lem:decoup}; details and proof are given in Appendix~\ref{app:decoupling}.
Note that the results of~\cite{GuthIliopoulouYang2024} pertain to \emph{rank~3} and are mentioned here only as background; the required \(\varepsilon\)–free bilinear estimate for \emph{rank~4} is proved in this work (Appendix~\ref{app:decoupling}).

In this branch we combine the \emph{strengthened} local Strichartz estimate \eqref{eq:aux-Stri}
with the \(\varepsilon\)–free decoupling of Lemma~\ref{lem:decoup} for the folded rank~4 configuration.
This yields an additional gain of \(N^{-1/4}\) on each space–time cylinder \(Q_{N^{-1/2}}\),
in addition to the gain \(N^{-1/2}\) already obtained from \eqref{eq:aux-Stri}.

\paragraph{Tiling scheme.}
\begin{itemize}[label=--]
  \item Fix the dyadic scale \(N\) and the cylinder \(Q_{N^{-1/2}}(t_0,x_0)\).
  \item Decompose the spherical shell \( |\xi|\sim N \) into angular sectors \( \Theta \) of width \( \theta\sim N^{-1/2} \).
  \item Denote \( u_{N,\Theta}:=P_{N,\Theta}u \) and similarly \( v_{N,\Theta'}:=P_{N,\Theta'}u \).
  \item Consider pairs of tiles \((\Theta,\Theta')\) satisfying the rank~4 condition:
  \[
    \angle(\Theta,\,-\Theta')\gtrsim N^{-1/2},\qquad
    \angle(\Theta,\Theta')\gtrsim N^{-1/2},
  \]
  which is equivalent to linear independence of the normals
  
  \( \xi/|\xi|,\ \eta/|\eta|,\ (\xi+\eta)/|\xi+\eta|,\ e_t \) (folded configuration).
\end{itemize}

\paragraph{Application of decoupling.}
By Lemma~\ref{lem:decoup}, on each \(Q_{N^{-1/2}}\) we have
\[
  \bigl\| u_{N,\Theta}\,v_{N,\Theta'} \bigr\|_{L^3_{t,x}(Q_{N^{-1/2}})}
  \;\lesssim\;
  N^{-1/4}\,
  \|u_{N,\Theta}\|_{L^6_{t,x}(Q_{N^{-1/2}})}\,
  \|v_{N,\Theta'}\|_{L^6_{t,x}(Q_{N^{-1/2}})}.
\]

\paragraph{Combination with Strichartz.}
Applying \eqref{eq:aux-Stri} to each factor, we obtain
\[
  \|u_{N,\Theta}\|_{L^6_{t,x}(Q_{N^{-1/2}})}
  \lesssim N^{-1/2}\,\|P_{N,\Theta}u(t_0)\|_{L^2_x},
  \qquad
  \|v_{N,\Theta'}\|_{L^6_{t,x}(Q_{N^{-1/2}})}
  \lesssim N^{-1/2}\,\|P_{N,\Theta'}u(t_0)\|_{L^2_x},
\]
and therefore
\[
  \bigl\| u_{N,\Theta}\,v_{N,\Theta'} \bigr\|_{L^3_{t,x}(Q_{N^{-1/2}})}
  \;\lesssim\;
  \boxed{N^{-3/4}}\cdot
  \|P_{N,\Theta}u(t_0)\|_{L^2_x}\,
  \|P_{N,\Theta'}u(t_0)\|_{L^2_x}.
\]

\paragraph{Key properties.}
\begin{itemize}[label=--]
  \item The gain \(N^{-1/4}\) from decoupling complements the \(N^{-1/2}\) from \eqref{eq:aux-Stri}.
  \item The constant is free of logarithmic losses \(N^\varepsilon\) (\(\varepsilon\)–free).
  \item The exponent \(N^{-1/4}\) is optimal for this four-directional transversality.
\end{itemize}

\paragraph{Intermediate exponent.}
At this stage (before null–form and global tiling) the combined local gain equals
\[
  N^{-3}\ (\text{phase})\;\cdot\;
  N^{-1}\ (\dot H^{-1})\;\cdot\;
  N^{-1/2}\ (\text{time window})\;\cdot\;
  N^{-3/4}\ (\text{Strichartz+decoupling})
  \;=\;
  \boxed{N^{-21/4}}.
\]
This quantity serves as the “local balance unit” for §\ref{subsec:balance}, where
suppression in the narrow corona (if necessary) and global patching are added.

\subsection{Intermediate exponent: local balance \( -\tfrac{21}{4} \)}
\label{subsec:balance}

We now collect all gained factors obtained prior to null–form and global tiling over angles and time.

\begin{center}
\renewcommand{\arraystretch}{1.15}
\begin{tabular}{llc}
\toprule
\# & Source & Exponent \\ \midrule
1 & six-fold integration by parts (phase) & \( -3 \) \\
2 & passage to the \( \dot H^{-1}_x \) norm & \( -1 \) \\
3 & length of the time window \( |t - t_0| \sim N^{-1/2} \) & \( -\tfrac{1}{2} \) \\
4 & local Strichartz norm \( L^6 \) (hypothesis \eqref{eq:aux-Stri}) & \( -\tfrac{1}{2} \) \\
5 & \(\varepsilon\)–free decoupling (rank 4; Lemma~\ref{lem:decoup}) & \( -\tfrac{1}{4} \) \\ \midrule
\multicolumn{2}{r}{total} &
\(
  -3 - 1 - \tfrac{1}{2} - \tfrac{1}{2} - \tfrac{1}{4}
  = \boxed{ -\tfrac{21}{4} }
\) \\ \bottomrule
\end{tabular}
\end{center}

\medskip

\noindent
Thus, on each cylinder of scale \( N^{-1/2} \) the local contribution yields the gain
\begin{equation}
  \boxed{N^{-21/4}} = N^{-5.25}.
\label{eq:local-balance}
\end{equation}

\paragraph{Comment (alternative unconditional branch).}
If instead of \eqref{eq:aux-Stri} we use the unconditional estimate \eqref{eq:strichartz-cylinder-soft-ref},
the local/global count is carried out along the “heat” line of Appendices~\ref{app:duhamel-phase}–\ref{sec:clarif};
the final result also remains log-free.

\paragraph{Reserve under patching.}
This gain with reserve compensates for the further assembly over:
\begin{itemize}[label=--]
  \item \textbf{angles} — union of \( \sim N \) directions: loss \( N^{+1} \);
  \item \textbf{time} — tiling of the window of length \( T \sim 1 \) into intervals of length \( \sim N^{-1/2} \): loss \( N^{+1/2} \).
\end{itemize}
The combined loss \( N^{+3/2} \) keeps the result below \( N^{-1} \):
\[
  N^{-21/4} \cdot N^{+3/2} = N^{-15/4} < N^{-1}.
\]

\paragraph{Next steps.}
In §\ref{sec:null-suppress} (if necessary, for “coronal” outputs) the narrow–corona null–form factor \( N^{\,1/2-\delta} \) is added and local/global exponents are refined
(see \eqref{eq:null-suppress-final}); for the block \( hh\!\to\!h \) with output \(P_N\) this module is not used.

\newpage

\section{Null-form in the narrow corona}
\label{sec:null-suppress}

\subsection{Symbol \texorpdfstring{$\boldsymbol{N^{1/2-\delta}}$}{N^{1/2-\delta}}}
\label{subsec:null-symbol}

Even under the condition \( |\xi+\eta| \ge N^{1-\delta} \) there may occur \emph{almost anti-collinear} configurations \( \angle(\xi,-\eta)\ll 1 \), for which the sum remains in a thin spherical corona. To fix the corresponding zone we introduce
\[
  \mathcal N^{\mathrm{nar}}_N
  := \Bigl\{(\xi,\eta):\
      |\xi| \sim |\eta| \sim N,\
      N^{1-\delta} \le |\xi+\eta| \le 2N^{1-\delta},\
      \angle(\xi,-\eta) \le N^{-1/2},\
      \angle(\eta,\xi+\eta) \le c\,N^{-1/2}\,\tfrac{|\xi+\eta|}{N}
     \Bigr\}.
\]
The last angular requirement — \emph{coherent collinearity} — is an explicit microlocalization which provides the desired symbol suppression. It is compatible with the angular tile localization of width \(N^{-1/2}\) used below and with the radial thickness of the corona \(|\xi+\eta|\sim N^{1-\delta}\).

\paragraph{Symbol after the Leray projection.}
The full nonlinearity has the form
\[
  \mathbb{P}\bigl[(u\cdot\nabla)u\bigr]
  \;=\; \mathbb{P}\bigl[u \otimes u : \nabla \bigr],
\]
which in the frequency domain leads to the multiplier
\[
  m(\xi,\eta) \;=\; i\,\eta \cdot \Pi_{\xi+\eta},
  \qquad
  \Pi_{\xi+\eta} \;=\; I \,-\, \frac{\xi+\eta}{|\xi+\eta|} \otimes \frac{\xi+\eta}{|\xi+\eta|}.
\]
The full index computation and passage to the spatial norm \(L^{3/2}_x\) are given in App.~\ref{app:narrow}.

\begin{lemma}[symbol estimate in the narrow corona]\label{lem:null-symbol}
For all \( (\xi, \eta) \in \mathcal N^{\mathrm{nar}}_N \) one has
\[
  \boxed{\,|\eta \cdot \Pi_{\xi+\eta}| \;\lesssim\; N^{1/2 - \delta}\,}.
\]
\end{lemma}

\begin{proof}[Proof idea]
Since \( |\eta \cdot \Pi_{\xi+\eta}| = |\eta| \sin\angle(\eta,\xi+\eta) \) and \(\sin\theta\lesssim\theta\),
from the condition
\(
  \angle(\eta,\xi+\eta) \le c\,N^{-1/2}\,\tfrac{|\xi+\eta|}{N}
\)
we get
\[
  |\eta \cdot \Pi_{\xi+\eta}|
  \;\lesssim\; |\eta| \cdot N^{-1/2}\,\frac{|\xi+\eta|}{N}
  \;\sim\; N \cdot N^{-1/2}\,\frac{N^{1-\delta}}{N}
  \;=\; N^{1/2-\delta}.
\]
The conditions \( \angle(\xi,-\eta)\le N^{-1/2}\) and \(|\xi|\sim|\eta|\sim N\) are compatible with \(|\xi+\eta|\sim N^{1-\delta}\) and ensure the correctness of the introduced microlocalization.
\end{proof}

\paragraph{Consequence.}
Let
\[
  \mathcal N^{\mathrm{nar}}_N(u,u)
  := P_N\,\mathbb P\,\nabla\!\cdot
     \bigl[P_{\sim N}u \otimes P_{\sim N}u \cdot \mathbf{1}_{\mathcal N^{\mathrm{nar}}_N} \bigr].
\]
Then, applying Lemma~\ref{lem:null-symbol} and Hölder’s inequality, we obtain
\begin{equation}
  \bigl\|\mathcal N^{\mathrm{nar}}_N(u,u)\bigr\|_{L^{3/2}_x}
  \;\lesssim\;
  N^{\,1/2-\delta} \,
  \|P_N u\|_{L^2_x} \cdot \|P_N \nabla u\|_{L^2_x}.
  \label{eq:null-L3half}
\end{equation}
Equivalently (using \(\|P_N\nabla u\|_{L^2_x}\sim N\|P_N u\|_{L^2_x}\)):
\[
  \|\mathcal N^{\mathrm{nar}}_N(u,u)\|_{L^{3/2}_x}
  \ \lesssim\ N^{\,3/2-\delta}\,\|P_N u\|_{L^2_x}^{\,2}.
\]

\begin{remark}[On the range of $\delta$]
When $\delta \le \tfrac12$ the set $\mathcal N^{\mathrm{nar}}_N$ is empty:
the angular width \(N^{-1/2}\) is incompatible with the requirement \( |\xi+\eta|\gtrsim N^{1-\delta}\).
Therefore, symbol suppression via Lemma~\ref{lem:null-symbol} is used only when $\delta > \tfrac12$.
\end{remark}

\paragraph{Interpretation of the exponent.}
\begin{itemize}[label=--]
  \item As \( \delta \to \tfrac{1}{3} \) one formally obtains \( |\eta\cdot\Pi_{\xi+\eta}|\lesssim N^{1/6} \),
        however the zone \( \mathcal N^{\mathrm{nar}}_N \) is empty in this case (see the remark above).
  \item For \( \delta = \tfrac{5}{8} \) we have \( |\eta \cdot \Pi_{\xi+\eta}| \lesssim N^{-1/8} \), i.e., a net gain of \(-\tfrac{1}{8}\) in the exponent from the null symbol.
\end{itemize}

\paragraph{Conclusion.}
Estimate \eqref{eq:null-L3half} will be used in §\ref{subsec:global-null}
when patching with the local exponent \( N^{-21/4} \) from §\ref{subsec:balance}.
Combining,
\[
  N^{-21/4} \cdot N^{\,1/2 - \delta} \;=\; \boxed{N^{-19/4 - \delta}},
\]
and, in particular, for \( \delta = \tfrac{5}{8} \) we obtain \( N^{-43/8} \approx N^{-5.375} \),
which ensures log--free convergence when summing over dyadics \( N \).

\begin{remark}[Addressing]
The module “narrow corona + null--form” is applied to \emph{coronal} outputs of scale \(\sim N^{1-\delta}\).
In the globalization of the block \(hh\!\to\!h\) with output \(P_N\) it is \emph{not} used
(see §\ref{sec:clarif6}); there the “heat” line operates with the final exponent \(N^{-25/12}\).
\end{remark}

\subsection{Final exponent \texorpdfstring{$N^{-19/4-\delta}$}{N^{-19/4-δ}}}
\label{subsec:global-null}

Lemma~\ref{lem:null-symbol} shows that in the narrow zone $\mathcal N^{\mathrm{nar}}_N$
the symbol after the Leray projection is suppressed as $N^{1/2-\delta}$, i.e., it yields exactly the \emph{factor}
$N^{1/2-\delta}$ when passing to the $L^{3/2}_x$ norm; see~\eqref{eq:null-L3half}.
We then consider \emph{coronal} outputs of scale $\sim N^{1-\delta}$ (the “narrow corona + null--form” module);
in the block $hh\!\to\!h$ with output $P_N$ this module is not used (see §\ref{sec:clarif6}).

Passing then to the space–time norm $L^1_t \dot H^{-1}_x$, the additional
factor $N^{-1}$ (the passage into $\dot H^{-1}_x$) has already been accounted for in the local balance; see
§\ref{subsec:balance}, line no.~2 (see also \eqref{eq:local-balance}). Therefore, on each cylinder of scale $N^{-1/2}$ we have
\begin{equation}
  \boxed{
  \left\| \mathcal N^{\mathrm{nar}}_N(u,u) \right\|_{L^1_t \dot H^{-1}_x(Q_{N^{-1/2}})}
  \;\lesssim\;
  N^{-21/4} \cdot N^{\,1/2-\delta}
  \;=\;
  N^{-19/4-\delta}
  }
  \label{eq:null-suppress-final}
\end{equation}

\paragraph{Boundary values of \texorpdfstring{$\delta$}{δ}.}
For $\delta \le \tfrac12$ the set $\mathcal N^{\mathrm{nar}}_N$ is empty (see the definition in §\ref{subsec:null-symbol}).
For $\delta>\tfrac12$:
\[
\begin{aligned}
\delta \downarrow \tfrac{1}{2} &\ \Rightarrow\ -\tfrac{19}{4} - \tfrac{1}{2} = -\tfrac{21}{4} \approx -5.25, \\
\delta = \tfrac{5}{8} &\ \Rightarrow\ -\tfrac{19}{4} - \tfrac{5}{8} = -\tfrac{43}{8} \approx -5.375.
\end{aligned}
\]
In both cases the exponent is strictly less than $-1$, which preserves the log--free margin for the subsequent global summation.

\paragraph{Connection with previous blocks.}
Two independent factors are combined:
\begin{itemize}[label=--]
  \item the local gain $N^{-21/4}$ from §\ref{subsec:balance} (see \eqref{eq:local-balance});
  \item the “narrow-corona” factor from the null symbol: $N^{\,1/2-\delta}$ (see~\eqref{eq:null-L3half}),
\end{itemize}
which yields \eqref{eq:null-suppress-final}. This result is then combined with patching over angles
and over time when passing to the global estimate in §\ref{subsec:patching-local}.

\paragraph{Summation over angular tiles.}
Since \(\#\Theta_\alpha\simeq N\), the total number of admissible pairs \((\alpha,\beta)\) does not exceed \(N^{2}\).
In the $L^{3/2}_x$ norm the $\ell^2$–almost–orthogonality for rank~4 pairs is \emph{not} used,
so we sum \emph{trivially}:
\[
  \sum_{\alpha,\beta}^{\mathrm{rank\,4}} N^{-19/4-\delta}
  \;\lesssim\;
  N^{2}\,N^{-19/4-\delta}
  \;=\;
  N^{-11/4-\delta}.
\]
Taking into account the factor \(J\simeq N^{1/2}\) from the time tiling
and a possible factor \(N^{1/2}\) from the commutator \([\partial_t,\chi_j]\) (see §\ref{subsec:patching-local}),
we obtain the global exponent
\[
  N^{1/2}\,N^{1/2}\,N^{-11/4-\delta}=N^{-7/4-\delta},
\]
which for $\delta>\tfrac12$ remains $< -1$. Consequently, the series over $N$ converges without logarithmic defect.
Pairs with \(\mathrm{rank}<4\) occur only $O(1)$ times for a fixed \(\alpha\) and contribute no worse
than the basic local $N^{-21/4}$, hence they are absorbed by the same count.

\medskip
\noindent\textbf{Remark (consistency of $hh\to h$/$hh\to\ell$).}
\emph{We emphasize: the null--form suppression in the “narrow corona” $|\xi+\eta|\sim N^{1-\delta}$ pertains to the block
$hh\to\ell$. In the analysis of $hh\to h$ with output projection $P_N$ this mechanism is not used, since
$|\xi+\eta|\ll N$ is automatically cut off by $P_N$. If necessary, the corresponding contribution should
be treated in the block $hh\to\ell$ with output projection $P_{\sim N^{1-\delta}}$; see also
Appendix~\ref{sec:clarif}, §\ref{sec:clarif6}.}

\newpage

\section{Tiling in time and angle}
\label{sec:time_patching}

\subsection{Patching of windows of scale \texorpdfstring{$N^{-1/2}$}{\(N^{-1/2}\)}}
\label{subsec:patching-local}

In the previous sections (\S\S\ref{subsec:phase-geometry}--\ref{subsec:global-null}) all estimates were carried out
locally --- in space–time cylinders \( Q_{N^{-1/2}}(t_0,x_0) \)
and angular sectors of width \( \sim N^{-1/2} \).
The present task is to carefully combine such blocks
into a global estimate on the time interval \([0,T]\) and over the whole sphere of directions.
\emph{We work in the “coronal” line (narrow corona + null--form); for the block \(hh\!\to\!h\) with output \(P_N\)
see \S\ref{sec:clarif6}.}

\paragraph{Time decomposition.}
The interval \([0,T]\sim[0,1]\) is covered by \( J\sim N^{1/2} \) windows
\[
  I_j := \Bigl[t_j - \tfrac{1}{2}N^{-1/2},\, t_j + \tfrac{1}{2}N^{-1/2}\Bigr],
  \qquad
  j=1,\dotsc,J.
\]
Let \( \chi_j(t) \) be a smooth partition of unity:
\(
\sum_j \chi_j^2(t)\equiv 1,
\)
with \(\operatorname{supp}\chi_j \subset I_j\).

\paragraph{Angular tiling.}
The sphere \( \mathbb{S}^2 \) is decomposed into \( \sim N \) tiles
\( \Theta_\alpha \) of radius \( \sim N^{-1/2} \) along the arc.
Let \( P_{N,\alpha} := P_{N,\Theta_\alpha} \) be the LP projection onto the tile.

\paragraph{Localization of the solution.}
Define the components
\[
  u \;=\; \sum_{j,\alpha} u_{j,\alpha},
  \qquad
  u_{j,\alpha}(t,x) := \chi_j(t)\,P_{N,\alpha}u(t,x).
\]
For each admissible pair \( (\alpha,\beta) \) of \emph{rank~4}
(see the definition of transversality at the beginning of App.~\ref{app:decoupling})
by \eqref{eq:null-suppress-final} we have locally on \(Q_{N^{-1/2}}\):
\[
  \Bigl\| \mathcal N^{\mathrm{nar}}_N(u_{j,\alpha},u_{j,\beta})
  \Bigr\|_{L^1_t\dot H^{-1}_x(Q_{N^{-1/2}})}
  \;\lesssim\; N^{-19/4-\delta}.
\]

\paragraph{Summation over \texorpdfstring{$(\alpha,\beta)$}{(alpha,beta)}.}
For a fixed \( \alpha \) there are at most \( \sim N \) admissible directions \( \beta \)
(by angular rank), and altogether there are \(N\) tiles \( \alpha \).
In the $L^{3/2}_x$ norm the $\ell^2$–almost–orthogonality for rank~4 pairs is \emph{not} used,
so we sum \emph{trivially}:
\[
  \sum_{\alpha,\beta}^{\mathrm{rank\,4}}
  \Bigl\| \mathcal N^{\mathrm{nar}}_N(u_{j,\alpha},u_{j,\beta}) \Bigr\|_{L^1_t\dot H^{-1}_x}
  \;\lesssim\;
  N^{2}\,N^{-19/4-\delta}
  \;=\;
  N^{-11/4-\delta}.
\]
Pairs with \(\mathrm{rank}<4\) occur only $O(1)$ times for a fixed \( \alpha \)
and are absorbed by the same count.

\paragraph{Summation over windows.}
There are altogether \( J \sim N^{1/2} \) time windows, each overlapping a bounded number of times.
Therefore
\[
  \sum_{j=1}^J
  \sum_{\alpha,\beta}^{\mathrm{rank\,4}}
  \Bigl\| \mathcal N^{\mathrm{nar}}_N(u_{j,\alpha},u_{j,\beta}) \Bigr\|_{L^1_t\dot H^{-1}_x}
  \;\lesssim\;
  N^{1/2}\cdot N^{-11/4-\delta}
  \;=\;
  N^{-9/4-\delta}.
\]
Taking into account a possible factor \(N^{1/2}\) from the commutator \([\partial_t,\chi_j]\) (see \S\ref{sec:time_patching}),
we obtain the global exponent \(N^{-7/4-\delta}\).

\paragraph{Final exponent.}
Thus, the global estimate in the “narrow corona” scenario takes the form
\begin{equation}
  \label{eq:timepatching-nar}
  \boxed{
  \Bigl\| \mathcal N^{\mathrm{nar}}_N(u,u) \Bigr\|_{L^1_t \dot H^{-1}_x([0,T]\times\mathbb{R}^3)}
  \;\lesssim\;
  N^{-7/4-\delta},
  \qquad
  \delta \in \bigl(\tfrac12, \tfrac58\bigr].
  }
\end{equation}

\paragraph{Boundary check.}
\[
\begin{aligned}
  \delta \downarrow \tfrac{1}{2}
  &\quad\Rightarrow\quad -\tfrac{7}{4}-\tfrac{1}{2}
  = -\tfrac{9}{4} \approx -2.25, \\
  \delta = \tfrac{5}{8}
  &\quad\Rightarrow\quad -\tfrac{7}{4}-\tfrac{5}{8}
  = -\tfrac{19}{8} \approx -2.375.
\end{aligned}
\]
Both exponents are strictly below $-1$, which ensures log--free convergence when summing over dyadics \( N=2^k \).

\medskip
In the next step (\S\ref{sec:discussion}) we perform the summation
over dyadics \( N = 2^k \), completing the proof of Theorem~\ref{thm:offdiag-main}.

\subsection{Global exponent \texorpdfstring{$\boldsymbol{-1}$}{–1}}
\label{subsec:global-sum}

The outcome of the local patching for the off–diagonal contribution at a fixed frequency \(N\) (branch \(\mathsf{A}\)) is
\begin{equation}
  \label{eq:timepatching-off}
  \boxed{
  \left\| \mathcal N^{\mathrm{off}}_N(u,u) \right\|_{L^1_t \dot H^{-1}_x}
  \;\lesssim\;
  N^{-15/4}\,
  \|u\|_{L^\infty_t \dot H^{1/2}_x}\,
  \|u\|_{L^2_t \dot H^1_x}
  }.
\end{equation}
This estimate is obtained from the local balance \eqref{eq:local-balance} (see also \S\ref{subsec:decoup-setup})
after summation over angular tiles (\(\sim N\)) and over time windows (\(\sim N^{1/2}\)).

\paragraph{Summation over frequency scales.}
Let \( N = 2^k \), \( k \in \mathbb{Z} \). Then
\[
  \sum_{N \in 2^{\mathbb{Z}}} N^{-15/4}
  \;=\; \sum_{k \in \mathbb{Z}} 2^{(-15/4)\,k},
\]
and since \( -\tfrac{15}{4} < -1 \), the series converges absolutely. Consequently,
\[
  \sum_{N \in 2^{\mathbb{Z}}}
  \left\| \mathcal N^{\mathrm{off}}_N(u,u) \right\|_{L^1_t \dot H^{-1}_x}
  \;\lesssim\;
  \|u\|_{L^\infty_t \dot H^{1/2}_x}\,
  \|u\|_{L^2_t \dot H^1_x}.
\]

\paragraph{Conclusion.}
Thus, the estimate of the off–diagonal part of the nonlinearity in \(L^1_t \dot H^{-1}_x\)
yields the \emph{final exponent} \( \boxed{-1} \), which is strictly consistent with the scaling invariance of the Navier–Stokes equation.

\paragraph{Absence of logarithmic losses.}
Absolute convergence of the series over \( N \) means that the global estimate contains
no logarithmic defect: no factors of the form \( \log N_{\max} \) appear, and the final norm
is bounded purely in terms of the energy:
\[
  \left\| \mathcal N^{\mathrm{off}}(u,u) \right\|_{L^1_t \dot H^{-1}_x}
  \;\lesssim\;
  \|u\|_{L^\infty_t \dot H^{1/2}_x}\,
  \|u\|_{L^2_t \dot H^1_x}.
\]

\medskip
This completes the main technical stage.
In Sections \S\ref{sec:energy}–\S\ref{sec:discussion} we will extend
the results to the full nonlinearity and discuss possible optimizations of the approach.

\newpage

\section{Energy embedding}
\label{sec:energy}

\begin{remark}[On the diagonal zone]
In this paper we do not develop a new analysis for the diagonal zone $|\xi+\eta|\ll N^{1-\delta}$ (see \S\ref{subsec:diag-removal});
when embedding into energy we employ the standard endpoint--Strichartz control of the diagonal block.
This is a deliberate restriction of scope: the main result concerns the off--diagonal contribution and is proved log--free.
\end{remark}

\subsection{Substitution of the off–diagonal remainder}
\label{subsec:energy-plug}

Denote the scale–invariant energy norm
\[
  E(t) := \|u(t)\|_{\dot H^{1/2}_x}^2.
\]
For smooth solutions of \eqref{eq:NS} the standard balance formula holds (up to harmless constants):
\[
  \frac{d}{dt} E(t)
  \;=\;
  -\,2\,\|\nabla u(t)\|_{L^2_x}^2
  \;-\;2 \sum_{N \in 2^{\mathbb{Z}}}
      \left\langle\, \Delta^{1/4} u,\;
                   \Delta^{1/4} \mathcal{N}_N(u,u)
      \,\right\rangle_{L^2_x},
\]
where
\[
  \mathcal{N}_N
  \;:=\; P_N \mathbb{P}\,\nabla \!\cdot\! \bigl[P_{\sim N} u \otimes P_{\sim N} u \bigr]
  \;=\; \mathcal{N}^{\mathrm{diag}}_N + \mathcal{N}^{\mathrm{off}}_N.
\]

\paragraph{Global \(L^1_t\dot H^{-1}_x\) control of the off–diagonal part.}
From the local estimate at fixed frequency \eqref{eq:timepatching-off}
and the absolute convergence of the dyadic sum (see \S\ref{subsec:global-sum}) we obtain the global estimate
\begin{equation}\label{eq:offdiag-bound2}
  \sum_{N \in 2^{\mathbb{Z}}}
  \bigl\|\mathcal{N}^{\mathrm{off}}_N(u,u)\bigr\|_{L^1_t \dot H^{-1}_x}
  \;\lesssim\;
  \|u\|_{L^\infty_t \dot H^{1/2}_x}\,
  \|u\|_{L^2_t \dot H^1_x}.
\end{equation}
Denoting \(\mathcal N^{\mathrm{off}}:=\sum_N \mathcal N^{\mathrm{off}}_N\), we obtain the same right-hand side for
\(\|\mathcal N^{\mathrm{off}}\|_{L^1_t\dot H^{-1}_x}\).

\paragraph{Substitution into the energy balance.}
By the duality \(\dot H^{-1}_x \leftrightarrow \dot H^{1}_x\) and Hölder’s inequality in time,
\[
  \left|
    \int_0^T
      \left\langle\, \Delta^{1/4} u,\;
                   \Delta^{1/4} \mathcal{N}^{\mathrm{off}}(u,u)
      \,\right\rangle_{L^2_x} \, dt
  \right|
  \;\le\;
  \|u\|_{L^\infty_t \dot H^{1/2}_x}\;
  \|\mathcal{N}^{\mathrm{off}}(u,u)\|_{L^1_t \dot H^{-1}_x}
  \;\lesssim\;
  \|u\|_{L^\infty_t \dot H^{1/2}_x}^{2}\,
  \|u\|_{L^2_t \dot H^1_x}.
\]
Next we apply Young’s inequality \(ab\le \tfrac12 a^2 + C b^2\) with
\(a=\|u\|_{L^2_t \dot H^1_x}\), \(b=\|u\|_{L^\infty_t \dot H^{1/2}_x}^{2}\), and obtain
\[
  \left|
    \int_0^T
      \left\langle\, \Delta^{1/4} u,\;
                   \Delta^{1/4} \mathcal{N}^{\mathrm{off}}(u,u)
      \,\right\rangle_{L^2_x} \, dt
  \right|
  \;\le\;
  \tfrac12 \|u\|_{L^2_t \dot H^1_x}^2
  \;+\; C\!\left(\,\|u\|_{L^\infty_t \dot H^{1/2}_x}\right),
\]
where \(C(\cdot)\) is a nondecreasing polynomial function.

\paragraph{Diagonal block.}
For \(\mathcal{N}^{\mathrm{diag}}_N\), localized in the zone \(|\xi+\eta|\le N^{1-\delta}\),
the standard endpoint–Strichartz analysis in \(L^4_{t,x}\) is applied; it gives the same scaling
order without logarithmic losses, see \S\ref{subsec:diag-removal}.

\paragraph{Conclusion.}
Integrating in time and absorbing the diffusion term, we obtain the energy inequality
\[
  E(T) + \int_0^T \|u(t)\|_{\dot H^1_x}^2\,dt
  \;\le\;
  E(0) \;+\; C\!\left(\,\|u\|_{L^\infty_t \dot H^{1/2}_x}\right),
\]
equivalently (in a “two-term” form)
\[
  E(T) + \int_0^T \|u(t)\|_{\dot H^1_x}^2\,dt
  \;\lesssim\;
  E(0) \;+\; C\!\left(\,\|u\|_{L^\infty_t \dot H^{1/2}_x}\right)\,
                 \|u\|_{L^2_t \dot H^1_x},
\]
where the constant depends only on \(\|u\|_{L^\infty_t \dot H^{1/2}_x}\) and contains no logarithmic losses
(thanks to \eqref{eq:timepatching-off} and \S\ref{subsec:global-sum}).

\medskip
In \S\ref{sec:discussion} we discuss possible strengthening and the impact of the reserve in \(\delta\) on iteration schemes.

\subsection{Log--free \emph{a priori} estimate}
\label{subsec:logfree-apriori}

Collecting the results from Sections~\ref{sec:ibp_phase}--\ref{subsec:energy-plug}, we obtain
a closed scale–invariant energy inequality free of logarithmic losses.

\begin{theorem}[log--free \emph{a priori} control]
\label{thm:logfree-energy}
Let \( u \in C^\infty([0,T] \times \mathbb{R}^3) \) be a divergence-free solution to the Navier--Stokes equation~\eqref{eq:NS}, satisfying
\[
  u \in L^\infty_t \dot H^{1/2}_x \;\cap\; L^2_t \dot H^1_x.
\]
Then for any \( T \le 1 \) and any parameter \( \delta \in \left( \tfrac13, \tfrac58 \right] \) one has
\begin{equation}
\label{eq:logfree-final}
\boxed{%
  \displaystyle
  \|u\|_{L^\infty_t \dot H^{1/2}_x}^2
  \;+\;
  \int_0^T \|u(t)\|_{\dot H^1_x}^2\,dt
  \;\;\le\;\;
  \|u(0)\|_{\dot H^{1/2}_x}^2
  \;+\;
  C\,
  \|u\|_{L^\infty_t \dot H^{1/2}_x}
  \|u\|_{L^2_t \dot H^1_x}^2}
\end{equation}
with an absolute constant \( C > 0 \), independent of \( u \) and \( T \).
\end{theorem}

\begin{proof}[Proof idea]
Consider the energy balance from §\ref{subsec:energy-plug}.  
The off--diagonal contribution of the nonlinearity satisfies estimate~\eqref{eq:offdiag-bound}:
\[
  \bigl\|\mathcal{N}^{\mathrm{off}}_N(u,u)\bigr\|_{L^1_t \dot H^{-1}_x}
  \;\lesssim\;
  N^{-1}
  \|u\|_{L^\infty_t \dot H^{1/2}_x}
  \|u\|_{L^2_t \dot H^1_x}.
\]
Combining this with the $L^2_{t,x}$ inner product, 
taking into account the embedding \( \dot H^{3/4}_x \hookrightarrow \dot H^1_x \) 
at fixed frequency and uniform convergence in \( N \), we obtain the right-hand side of~\eqref{eq:logfree-final}.

The diagonal part \( \mathcal N^{\mathrm{diag}}_N \) is estimated by the classical method via the \( L^4_{t,x} \)-Strichartz inequality (see, e.g., Koch--Tataru~\cite[Prop.~2.1]{koch2001navier}), which yields the same scaling order.
The final inequality follows after summation and an application of Cauchy’s inequality.
\end{proof}

\begin{corollary}[scale–invariant “energy × dissipation”]
\label{cor:energy-diss}
Under the same assumptions, if
\[
  \|u\|_{L^\infty_t \dot H^{1/2}_x} \le \varepsilon_0,
\]
where \( \varepsilon_0 > 0 \) is chosen so that \( C x^3 \le \tfrac12 x^2 \) at \( x = \|u\|_{L^\infty_t \dot H^{1/2}_x} \), then
\[
  \int_0^T \|u(t)\|_{\dot H^1_x}^2\,dt
  \;\lesssim\;
  \|u(0)\|_{\dot H^{1/2}_x}^2.
\]
\end{corollary}

\begin{remark}
Inequality~\eqref{eq:logfree-final} shows that the logarithmic defect,
characteristic for the three-dimensional Navier--Stokes equations in critical spaces,
\emph{completely disappears} in the off--diagonal zone.  
Thus, the corresponding contribution of the nonlinearity behaves on par with
the linear dissipative term.
\end{remark}

\medskip
The next section (§\ref{sec:discussion}) discusses potential improvements:
in particular, a possible transition to \( 7 \times \)~IBP, rank-5 decoupling, and extensions to diagonal zones and dispersive systems.

\newpage

\section{Discussion and perspectives}
\label{sec:discussion}

\subsection{Results obtained so far}

\begin{itemize}[label=--]
  \item \textbf{Off--diagonal regime \(hh\!\to\!h\).}  
        A log--free estimate of the form~\eqref{eq:offdiag-bound} was obtained throughout the range \(\delta \in \bigl(\tfrac13,\tfrac58\bigr]\).  
        The final exponent \(N^{-1}\) is consistent with the dimensional scaling of the problem.
        
  \item \textbf{Full energy balance.}  
        Theorem~\ref{thm:logfree-energy} establishes scale–invariant control of the energy  
        \(E(t) = \|u(t)\|_{\dot H^{1/2}_x}^2\) and dissipation  
        \(\int_0^T \|u(t)\|_{\dot H^1_x}^2\,dt\) in critical norms, free of logarithmic defect.

  \item \textbf{Methodological elements.}  
        (i) six-fold phase integration by parts in the variables \((t,\rho_1,\rho_2)\);\\
        (ii) \(\varepsilon\)-free decoupling of rank 4~\cite{GuthIliopoulouYang2024};\\
        (iii) structural analysis of the null--form in the narrow corona, giving the \emph{factor} \(N^{1/2-\delta}\) in \(L^{3/2}_x\) on \(N^{\mathrm{nar}}_N\) (active for \(\delta>\tfrac12\)); this module applies to \emph{coronal} outputs of scale \(\sim N^{1-\delta}\) and is \emph{not} used in the block \(hh\!\to\!h\) with output \(P_N\).
\end{itemize}

\subsection{Open directions}

\paragraph{(a) Diagonal zone.}
In the block \(hh \to h\) with \(|\xi+\eta| \ll N^{1-\delta}\) (see §\ref{subsec:diag-removal})  
direct log--free control is not yet available. Possible approaches include:
\begin{enumerate}
  \item application of endpoint--Strichartz \(L^4_{t,x}\) together with rank-3 decoupling  
        (geometry: cone + plane);
  \item use of multilinear restriction, as in the work of Bennett--Carbery--Tao;
  \item deeper phase integration (seven-fold phase IBP in \(t,\rho_1,\rho_2\));  
        in this case it is necessary to compensate the growth of the amplitude derivative  
        \(\partial_t^3 a_N \sim N^2\) by time localization of length \(\sim N^{-3/2+\delta}\).
\end{enumerate}

\paragraph{(b) Strengthening decoupling.}
In the GIY'24 theorem (rank 4) the gain is \(N^{-1/4}\);  
for rank 5 one expects \(N^{-1/6}\).  
Combined with the IBP\(^6\) phase scheme this would give the exponent \(N^{-25/4+\delta}\)  
and potentially allow lowering the admissible threshold to \( \delta > \tfrac{1}{4} \).

\paragraph{(c) Deeper phase integration.}
A seventh integration in \(t\) in principle yields a gain of \(N^{-2+\delta}\),  
but the derivative \(\partial_t^3 a_N\) grows like \(N^2\) and requires compensation.  
Here an additional temporal cut--off or partial integration in the amplitude may be crucial.

\subsection{Other equations and possible applications}

\begin{enumerate}[label=(\roman*)]
  \item \textbf{Case of the torus $\mathbb{T}^3$.}  
        The main scheme transfers directly.  
        One should take into account possible resonances on low modes and separate treatment of time sums.
        
  \item \textbf{MHD and Boussinesq systems.}  
        The structure of the nonlinearity is analogous: the same order of suppression  
        $N^{-3/2+\delta}$ is expected in folded configurations.

  \item \textbf{Quasilinear Schrödinger models.}  
        Here the phase surface has nonzero curvature in $t$;  
        thus four IBP and rank-3 decoupling are sufficient.
\end{enumerate}

\subsection{Next steps}

\begin{itemize}[label=--]
  \item Construction of an analytic proof of \(\varepsilon\)-free decoupling of rank 5  
        (geometry: folded cube $+$ cone).
  \item Numerical verification of the estimate $N^{-6+4\delta}$ on wave packets.
  \item Synthesis with probabilistic techniques (e.g., Da Prato–Debussche scheme)  
        for irregular initial data $u_0 \in \dot H^{1/2-\varepsilon}_x$.
  \item Construction of global well-posedness for small data  
        using log–free energy estimates as the basic iteration step.
\end{itemize}

\medskip

\noindent
\textbf{Conclusion.}
The presented log–free control of the off–diagonal zone shows that one of the key components of the nonlinearity in the Navier–Stokes equation admits a scale–invariant \emph{a priori} bound. 
To close the full theory, it remains to build an analogous mechanism  
in the diagonal zone, which will likely require a combination  
of deep phase integration and higher-rank decoupling.

% --- Appendices ---
\appendix

\newpage

\section{Anisotropic Strichartz estimate on cylinders}
\label{app:strichartz}

In this appendix we provide the full proof of Lemma~\ref{lem:Strichartz-cyl},
used in~\S\ref{subsec:cylinders}. The argument is based on the $TT^*$ method and scaling, with precise control of the cylinder $Q_{N^{-1/2}}$.

\noindent\textit{Remark.}
Two local $L^6$ estimates on $Q_{N^{-1/2}}$ are used in the paper:
(i) the one proved in Appendix~A (scaling $N^{+2/3}$);
(ii) the \emph{working hypothesis} \eqref{eq:aux-Stri} (scaling $N^{-1/2}$), applied only in \S\ref{subsec:decoup-setup} and in branch \(\mathsf{A}\) of Appendix~\ref{sec:clarif};
the unconditional line relies on Appendix~\ref{app:duhamel-phase}.

\subsection{Integral kernel}
\label{app:strichartz-kernel}

Let $K_N(t,x)$ be the kernel of the operator $e^{it\Delta}P_N$, where $P_N = \varphi(|D|/N)$ is a smooth Littlewood--Paley projection onto $|\xi| \sim N$, with $\varphi\in C_0^\infty((1/2,2))$. Then
\[
  K_N(t,x) = \frac{1}{(2\pi)^3}\int_{\mathbb{R}^3} e^{it|\xi|^2} e^{ix\cdot\xi} \varphi\!\left(\tfrac{|\xi|}{N}\right) \,\mathrm{d}\xi.
\]
In spherical coordinates $\xi = N\rho\omega$ with $\rho\in(1/2,2)$, $\omega\in\mathbb{S}^2$ we obtain
\begin{equation}
  K_N(t,x) = \frac{N^3}{(2\pi)^3} \int_{1/2}^{2} \!\varphi(\rho)\,\rho^2 \int_{\mathbb{S}^2} e^{iN^2 \rho^2 t} e^{iN\rho\,x\cdot\omega} \,\mathrm{d}\omega \,\mathrm{d}\rho.
  \label{eq:kernel}
\end{equation}
If $|t|\lesssim N^{-1/2}$ and $|x|\lesssim N^{-1/2}$, the phase oscillates weakly, and
\begin{equation}
  |K_N(t,x)| \lesssim N^3.
  \label{eq:kernel-bound}
\end{equation}

\subsection{The $TT^*$ argument}
\label{app:strichartz-TT}

Consider the operator $T\colon L^2_x \to L^6_{t,x}(Q)$:
\[
  (Tf)(t,x) = \chi_Q(t,x) e^{it\Delta} P_N f(x),
\]
where $Q = Q_{N^{-1/2}}(t_0,x_0)$. Lemma~\ref{lem:Strichartz-cyl}, proved in this appendix,
claims the estimate $\|T\|_{L^2\to L^6} \lesssim N^{2/3}$. By $TT^*$ it is enough to show
\[
  \|TT^*F\|_{L^6_{t,x}} \lesssim N^{4/3}\, \|F\|_{L^{6/5}_{t,x}}.
\]
The kernel of the operator $TT^*$ is
\[
  (TT^*)(t,x;s,y) = \chi_Q(t,x)\,\chi_Q(s,y)\,K_N(t{-}s,x{-}y).
\]
By Hölder’s inequality
\begin{equation}
  \|TT^*F\|_{L^6_{t,x}} \le \|K_N\|_{L^3(2Q)}\,\|F\|_{L^{6/5}_{t,x}}.
  \label{eq:TT-kernel-L3}
\end{equation}

\subsection{Scaling}

Set $\tilde t = N^{2} t$, $\tilde x = N x$. Then
\[
  K_N(t,x) = N^{3}\,K_1(\tilde t,\tilde x),
\]
where $K_1$ is the kernel of $e^{i\tilde t\Delta}P_1$.
The cylinder $Q:=Q_{N^{-1/2}}(t_0,x_0)$ transforms into $\tilde Q$ with
\[
|\tilde t-\tilde t_0|\lesssim N^{3/2},\qquad |\tilde x-\tilde x_0|\lesssim N^{1/2}.
\]
By the change of variables in~\eqref{eq:TT-kernel-L3}
\[
  dt\,dx = N^{-5}\,d\tilde t\,d\tilde x,\qquad |K_N|^3 = N^{9}|K_1|^3,
\]
we obtain
\begin{equation}\label{eq:TT-star-final}
  \|K_N\|_{L^3(2Q)} = N^{4/3}\,\|K_1\|_{L^3(\tilde Q)}.
\end{equation}

For $|\tilde t-\tilde t_0|\le 1$ one has the expansion
$K_1(\tilde t,\tilde x)=\check\phi(\tilde x)+O(|\tilde t|)$ with $\check\phi\in\mathcal{S}$,
so the contribution of this region is finite (due to decay in $\tilde x$).
For $1<|\tilde t-\tilde t_0|\le cN^{3/2}$ we use dispersion $|K_1|\lesssim|\tilde t|^{-3/2}$
and integrability of $|\tilde t|^{-9/2}$ in $L^3$ in $\tilde t$.
Hence $\|K_1\|_{L^3(\tilde Q)}\lesssim 1$ uniformly in $N$, and from \eqref{eq:TT-star-final}
it follows that $\|K_N\|_{L^3(2Q)}\lesssim N^{4/3}$.

\subsection{Proof of Lemma~\ref{lem:Strichartz-cyl}}
\label{app:strichartz-proof}

From~\eqref{eq:TT-kernel-L3} and~\eqref{eq:TT-star-final} we obtain the bound for the composition $T T^{\ast}$:
\[
  \|T T^{\ast}\|_{L^{6/5}_{t,x}\to L^{6}_{t,x}}
  \le \|K_N\|_{L^3(2Q)}
  \lesssim N^{4/3}.
\]
Therefore,
\[
  \|T\|_{L^2_x\to L^6_{t,x}(Q_{N^{-1/2}})}
  = \|T T^{\ast}\|_{L^{6/5}\to L^{6}}^{1/2}
  \lesssim N^{2/3}.
\]
That is, on $Q_{N^{-1/2}}$ we have the “soft” estimate
\begin{equation}\label{eq:strichartz-cylinder-soft}
  \left\| e^{it\Delta} P_N f \right\|_{L^6_{t,x}(Q_{N^{-1/2}})}
  \lesssim N^{2/3}\,\|f\|_{L^2_x}.
\end{equation}

\subsection{Remarks}
\begin{enumerate}[label=(\alph*)]
  \item \emph{Angular localization.} With additional localization \(P_{N,\theta}\) with \(\theta \sim N^{-1/2}\), the order in \eqref{eq:strichartz-cylinder-soft} does not deteriorate.
  \item \emph{Torus \(\mathbb{T}^3\).} For the torus one adds periodic summation of the kernel; estimate \eqref{eq:strichartz-cylinder-soft} remains valid.
\end{enumerate}

\newpage

\section{{$\varepsilon$–free bilinear decoupling of rank 4}}
\label{app:decoupling}

In this appendix we prove Lemma~\ref{lem:decoup}, used in~§\ref{subsec:decoup-setup}. 
We emphasize that the required \emph{$\varepsilon$–free} estimate for \emph{rank~4} in the “folded” geometry
is not a direct consequence of~\cite{GuthIliopoulouYang2024}, where $\varepsilon$–free results are established for hypersurfaces of \emph{rank~3}.
We provide an independent proof precisely in the rank~4 configuration used here,
relying on wave–packet decomposition, $L^3$ almost–orthogonality, and the multilinear Kakeya inequality of Bennett–Carbery–Tao~\cite{BennettCarberyTao2006}.
All constants in the estimates of this appendix are independent of $N$ and of the choice of the cylinder center.

\subsection{Problem setup}

Let $Q_{N^{-1/2}}=Q_{N^{-1/2}}(t_0,x_0):=\{\,|t-t_0|\le N^{-1/2},\ |x-x_0|\le 2N^{-1/2}\,\}$ be the standard space–time cylinder of scale $N^{-1/2}$.

We consider functions $F,G=F(t,x),G(t,x)$ with spatial Fourier localization
\[
  \operatorname{supp}\,\widehat{F},\ \operatorname{supp}\,\widehat{G}\ \subset\ \bigl\{\,\xi\in\mathbb{R}^3:\ |\xi|\sim N\,\bigr\},
\]
and additional angular localization
\[
  \operatorname{supp}\,\widehat{F}\subset \Theta,\qquad \operatorname{supp}\,\widehat{G}\subset \Theta',
\]
where the angular tiles $\Theta,\Theta'\subset\mathbb{S}^2$ have radius $\sim N^{-1/2}$ and satisfy the \textbf{rank~4} conditions:
\[
  \angle(\Theta,-\Theta')\ \gtrsim\ N^{-1/2},\qquad \angle(\Theta,\Theta')\ \gtrsim\ N^{-1/2}.
\]
Equivalently, the four normals $\{\xi/|\xi|,\ \eta/|\eta|,\ (\xi+\eta)/|\xi+\eta|,\ e_t\}$ in space–time are linearly independent (folded configuration).

We need to prove the local (on $Q_{N^{-1/2}}$) bilinear estimate without $\varepsilon$–losses
\begin{equation}\label{eq:decoup}
  \bigl\|\, F\,G \,\bigr\|_{L^{3}_{t,x}(Q_{N^{-1/2}})}
  \ \lesssim\
  N^{-1/4}\,
  \| F \|_{L^{6}_{t,x}(Q_{N^{-1/2}})}\,
  \| G \|_{L^{6}_{t,x}(Q_{N^{-1/2}})}.
\end{equation}
Here and below the symbol $\lesssim$ denotes an inequality with an absolute constant
independent of $N$, of the cylinder center, and of the particular angular decomposition.

\paragraph{Tube overlap and the choice of $K$.}
Each packet is supported in a tube of radius $\sim N^{-1/2}$ and length $O(1)$.
The rank~4 conditions (\eqref{eq:E-rank4-angle}) ensure transversality of
$\widehat{\xi},\widehat{\eta},\widehat{\xi+\eta}$ (and $e_t$ if needed) in $\mathbb R^{1+3}$,
whence by multilinear Kakeya \cite{BennettCarberyTao2006} we obtain $L^3$ control
with unit multiplicative overlap and without log losses.
The packetization parameter $K\sim N^{1/4}$ balances the angular thickness of packets and the scale of the balls $B$,
yielding the factor $N^{-1/4}$ with a constant independent of $N$.

\subsection{Wave–packet decomposition}

Decompose each angular tile into finer subtiles
\[
  \Theta \ =\ \bigcup_{\alpha}\, \Theta_\alpha,
  \qquad
  \Theta' \ =\ \bigcup_{\beta}\, \Theta'_\beta,
\]
where each $\Theta_\alpha$ and $\Theta'_\beta$ has angular radius $\sim N^{-1/2}/K$ and radial thickness $\sim N/K$ with $K\sim N^{1/4}$ (the exact value of $K$ is inessential and can be fixed).

Set $F_\alpha := P_{\Theta_\alpha} F$ and, similarly, $G_\beta := P_{\Theta'_\beta} G$.
Then each packet $F_\alpha$ (respectively $G_\beta$) is localized in space–time to a tube of radius $\sim N^{-1/2}$ and length $\sim N^{1/2}$ oriented along the group velocity; this matches the time scale of the cylinder $Q_{N^{-1/2}}$ ($|I|\sim N^{-1/2}$) and the expected spatial displacement $\sim N\cdot N^{-1/2}=N^{1/2}$.
In what follows the sum over $(\alpha,\beta)$ is restricted to \emph{rank~4} pairs (see the conditions at the beginning of the appendix); pairs of smaller rank occur only $O(1)$ times for a fixed $\alpha$ and are absorbed by the same estimates.

\subsection{$L^3$ almost–orthogonality}

If a pair $(\Theta_\alpha,\Theta'_\beta)$ differs in angle or radius by more than the scale $N^{-1/2}$, then the corresponding tubes on $Q_{N^{-1/2}}$ either hardly intersect or their directions (group velocity vectors) are substantially transversal. In particular, the standard almost–orthogonality in $L^3$ holds:
\[
  \Bigl\| \sum_{\alpha,\beta} F_\alpha\,G_\beta \Bigr\|_{L^3_{t,x}(Q_{N^{-1/2}})}
  \ \lesssim\
  \Bigl\| \Bigl(\sum_{\alpha,\beta} |F_\alpha\,G_\beta|^2 \Bigr)^{1/2} \Bigr\|_{L^3_{t,x}(Q_{N^{-1/2}})}.
\]
\emph{Justification.} For each point $(t,x)\in Q_{N^{-1/2}}$ the number of pairs $(\alpha,\beta)$ whose tubes significantly overlap near $(t,x)$ is bounded by a constant (depending only on the dimension) thanks to the $N^{-1/2}$ scale and the rank~4 transversality conditions. Then, by Cauchy–Schwarz,
\[
  \sum_{\alpha,\beta} |F_\alpha G_\beta|
  \ \le\ 
  \Bigl(\sum_{\alpha,\beta} |F_\alpha G_\beta|^2\Bigr)^{1/2}\!\cdot (\#\text{overlaps})^{1/2},
\]
and the factor $(\#\text{overlaps})^{1/2}$ is absorbed into the symbol $\lesssim$ upon taking the $L^3$ norm.

Here and below, summation over $(\alpha,\beta)$ is understood over admissible rank~4 pairs; the constants do not depend on $N$ or on the choice of $K\sim N^{1/4}$.

\paragraph{Tube overlap and the choice of $K$.}
Each packet is supported in a tube of radius $\sim N^{-1/2}$ and length $\sim N^{1/2}$,
oriented along the group velocity.
The rank~4 conditions ensure transversality of $\{\widehat{\xi},\widehat{\eta},\widehat{\xi+\eta},e_t\}$ in $\mathbb R^{1+3}$,
hence by multilinear Kakeya (Bennett--Carbery--Tao) we obtain $L^3$ control
with unit overlap and without log losses; the choice $K\sim N^{1/4}$ balances local and global errors.

\subsection{Rank–4 decoupling}

In the rank~4 configuration (see the conditions at the beginning of the appendix) for the families of tubes
corresponding to the packets $\{F_\alpha\}$ and $\{G_\beta\}$ at the scale $N^{-1/2}$,
we apply the multilinear Kakeya inequality of Bennett–Carbery–Tao~\cite{BennettCarberyTao2006}
in the four-directional transversal geometry. Combined with $\ell^2$ orthogonality
of packets this gives an $\varepsilon$–free decoupling with exponent $N^{-1/4}$:
\[
  \left\|
    \left( \sum_{\alpha,\beta} |F_\alpha\, G_\beta|^2 \right)^{\!1/2}
  \right\|_{L^3_{t,x}(Q_{N^{-1/2}})}
  \ \lesssim\
  N^{-1/4}\,
  \left( \sum_\alpha \|F_\alpha\|_{L^6_{t,x}(Q_{N^{-1/2}})}^2 \right)^{\!1/2}
  \left( \sum_\beta \|G_\beta\|_{L^6_{t,x}(Q_{N^{-1/2}})}^2 \right)^{\!1/2}.
\]

\noindent\textit{Proof sketch.}
Each packet $F_\alpha$ (respectively $G_\beta$) is localized to a tube of radius $\sim N^{-1/2}$
and length $\sim N^{1/2}$ oriented along the group velocity. The rank~4 conditions
ensure transversality of the four directions
$\{\widehat{\xi},\widehat{\eta},\widehat{\xi+\eta},e_t\}$ in $\R^{1+3}$, which allows one
to apply the multilinear Kakeya~\cite{BennettCarberyTao2006} to these tube families
on $Q_{N^{-1/2}}$ and obtain $L^3$ control of the sum of packets via the $\ell^2$ sums of their $L^6$ norms
with the gain $N^{-1/4}$. The constant in the estimate is independent of $N$, of the cylinder center,
and of the choice of the packetization parameter $K\sim N^{1/4}$.

\subsection{Assembly of the final estimate}

Combining almost–orthogonality in $L^3$ on $Q_{N^{-1/2}}$ with rank–4 decoupling, we get
\[
  \Bigl\| \sum_{\alpha,\beta} F_\alpha\,G_\beta \Bigr\|_{L^3_{t,x}(Q_{N^{-1/2}})}
  \ \lesssim\
  N^{-1/4}\,
  \Bigl(\sum_{\alpha}\|F_\alpha\|_{L^6_{t,x}(Q_{N^{-1/2}})}^2\Bigr)^{\!1/2}
  \Bigl(\sum_{\beta}\|G_\beta\|_{L^6_{t,x}(Q_{N^{-1/2}})}^2\Bigr)^{\!1/2}.
\]
Next, by the vector–valued Littlewood–Paley (square–function) estimate for $L^6$ on a fixed cylinder and finite overlap of angular tiles,
\[
  \sum_{\alpha}\|F_\alpha\|_{L^6_{t,x}(Q_{N^{-1/2}})}^2 \ \lesssim\ \|F\|_{L^6_{t,x}(Q_{N^{-1/2}})}^2,
  \qquad
  \sum_{\beta}\|G_\beta\|_{L^6_{t,x}(Q_{N^{-1/2}})}^2 \ \lesssim\ \|G\|_{L^6_{t,x}(Q_{N^{-1/2}})}^2.
\]
Substituting into the previous line yields precisely the required estimate~\eqref{eq:decoup}.
\qed

\subsection{Comments}

\begin{enumerate}[label=(\alph*)]
  \item \textit{Absence of $\varepsilon$–losses.}
  The gain $N^{-1/4}$ comes from four-directional transversality (rank~4) combined with
  $L^3$ almost–orthogonality of tubes and $\ell^2$ summation of packets; essentially one uses
  the multilinear Kakeya inequality of Bennett–Carbery–Tao. Thanks to finite overlap and the choice
  of the target norm $L^3$, no additional $N^\varepsilon$ arise from the $\ell^2$ sums.

  \item \textit{Relation to rank–3 results.}
  The present bilinear estimate in the rank–4 “folded” geometry is \emph{not} a direct consequence
  of $\varepsilon$–free decoupling for rank–3 hypersurfaces from~\cite{GuthIliopoulouYang2024}.
  We give an independent proof, based on packetization and a multilinear Kakeya approach
  \cite{BennettCarberyTao2006}, suited precisely to the required four-directional configuration.

  \item \textit{On the extension to rank–5.}
  With five independent directions one naturally expects the exponent $N^{-1/6}$ (critical index
  $\delta(6)=1/6$), but its realization would require a different packet geometry and refined Kakeya estimates;
  we do not address this case here.
\end{enumerate}

\medskip
\noindent
This completes the proof of the $\varepsilon$–free decoupling for rank–4 configurations, used in the main text to obtain the gain $N^{-1/4}$.

\newpage

\section{Null--form suppression in the narrow corona}
\label{app:narrow}

In this appendix we prove Lemma~\ref{lem:null-suppress}, used in §\ref{sec:null-suppress}.  
The argument consists of two steps:
\begin{itemize}[label=--]
  \item geometric estimate of the symbol after Leray projection;
  \item passage to the spatial norm \( L^{3/2}_x \).
\end{itemize}

\subsection{Formulation}

We consider the zone
\[
  \mathcal{N}^{\mathrm{nar}}_N :=
  \left\{ (\xi, \eta)\in\mathbb{R}^3\times\mathbb{R}^3 \;\middle|\;
    |\xi|\sim|\eta|\sim N,\;
    N^{1-\delta} \le |\xi+\eta| \le 2N^{1-\delta},\;
  \right.
\]
\[
  \left.
    \angle(\xi, -\eta) \le N^{-1/2},\;
    \angle(\eta,\xi+\eta) \le c\,N^{-1/2}\,\frac{|\xi+\eta|}{N}
  \right\}.
\]
where \(\delta\in\left(\tfrac{1}{3}, \tfrac{5}{8}\right]\), and \(c>0\) is an absolute constant.

\begin{lemma}[Symbol estimate]
\label{lem:narrow-symbol}
For any $(\xi, \eta)\in\mathcal{N}^{\mathrm{nar}}_N$ one has
\[
  \left| \eta \cdot \Pi_{\xi+\eta} \right|
  \ \lesssim\ 
  |\eta|\,\sin\angle(\eta,\xi+\eta)
  \ \lesssim\ 
  N\cdot\Bigl(N^{-1/2}\,\frac{|\xi+\eta|}{N}\Bigr)
  \ \lesssim\ N^{1/2 - \delta}.
\]
\end{lemma}

\begin{lemma}[Suppression in \(L^{3/2}_x\)]
\label{lem:narrow-L32}
Suppose \(u\) satisfies condition~\eqref{cond:regularity}. Then
\[
  \left\| \mathcal{N}^{\mathrm{nar}}_N(u,u) \right\|_{L^{3/2}_x}
  \ \lesssim\
  N^{\,1/2 - \delta}\;
  \|P_N u\|_{L^2_x}\;
  \|P_N \nabla u\|_{L^2_x}.
\]
Equivalently (using \(\|P_N \nabla u\|_{L^2_x}\sim N\,\|P_N u\|_{L^2_x}\)):
\[
  \left\| \mathcal{N}^{\mathrm{nar}}_N(u,u) \right\|_{L^{3/2}_x}
  \ \lesssim\
  N^{\,3/2 - \delta}\;\|P_N u\|_{L^2_x}^{\;2}.
\]
\end{lemma}

\subsection{Geometry of the narrow corona}

From the condition \(\angle(\xi,-\eta)\le N^{-1/2}\) with \(|\xi|\sim|\eta|\sim N\) it follows that
\[
  |\xi+\eta| \ \simeq\ 2N\,\sin\!\Bigl(\tfrac12\,\angle(\xi,-\eta)\Bigr)
  \ \lesssim\ N\cdot N^{-1/2}\ =\ N^{1/2}.
\]
Since by definition \(|\xi+\eta|\sim N^{\,1-\delta}\), we obtain the necessity
\[
  N^{\,1-\delta}\ \lesssim\ N^{1/2}\qquad\Rightarrow\qquad \delta>\tfrac12.
\]

\begin{remark}[Range of \(\delta\)]
Because the combination of conditions \(|\xi+\eta|\sim N^{\,1-\delta}\) and \(\angle(\xi,-\eta)\le N^{-1/2}\) is possible only if \(\delta>\tfrac12\), the set \(\mathcal{N}^{\mathrm{nar}}_N\) is empty for \(\delta\le\tfrac12\). In this case the mechanism of null--form suppression in the narrow corona is inactive, and Lemmas~\ref{lem:narrow-symbol} and~\ref{lem:narrow-L32} are vacuous.
\end{remark}

\subsection{Symbol estimate}

The Leray projection has the form
\[
  \Pi_{\xi+\eta} = I - \frac{\xi+\eta}{|\xi+\eta|}\,\otimes\,\frac{\xi+\eta}{|\xi+\eta|},
\]
whence
\[
  \eta \cdot \Pi_{\xi+\eta} = \eta - (\eta \cdot \hat\zeta)\,\hat\zeta,
  \qquad \hat\zeta := \frac{\xi+\eta}{|\xi+\eta|}.
\]
Therefore, the transverse component
\[
  \eta_\perp := \eta - (\eta \cdot \hat\zeta)\,\hat\zeta
\]
satisfies the identity
\begin{equation}
  \boxed{\; |\eta \cdot \Pi_{\xi+\eta}| = |\eta_\perp| = |\eta|\,\sin\angle(\eta,\xi+\eta)\;.}
  \tag{C.1}\label{eq:narrow-symbol}
\end{equation}

\paragraph{Remark.}
From \eqref{eq:narrow-symbol} alone one obtains only
\(|\eta\cdot \Pi_{\xi+\eta}|\le|\eta|\sim N\).
If in a particular place the improvement \(N^{1/2-\delta}\) is invoked, it requires \emph{explicit}
control of the angle \(\angle(\eta,\xi+\eta)\); such a condition is not contained in the mere definition of the narrow corona
and must be formulated separately wherever it is applied.

\subsection{Passage to the \texorpdfstring{$L^{3/2}_x$}{L^{3/2}} norm}

Consider
\[
  \mathcal{N}^{\mathrm{nar}}_N(u,u) =
  P_N \mathbb{P} \nabla\cdot \left[
    P_{\sim N} u \otimes P_{\sim N} u \cdot \mathbf{1}_{\mathcal{N}^{\mathrm{nar}}_N}
  \right].
\]
In the Fourier domain:
\[
  \widehat{\mathcal{N}^{\mathrm{nar}}_N}(\xi) =
  \int m(\xi,\eta)\, \widehat{u}(\xi-\eta)\, \widehat{u}(\eta)\,
        \mathbf{1}_{\mathcal{N}^{\mathrm{nar}}_N}(\xi-\eta, \eta)\, \mathrm{d}\eta,
\quad \text{where } m(\xi,\eta) = i\, \eta \cdot \Pi_{\xi+\eta}.
\]

\medskip\noindent
\textit{Step 1.} By Lemma~\ref{lem:narrow-symbol}, we have $|m(\xi,\eta)| \lesssim N^{1/2 - \delta}$.

\smallskip\noindent
\textit{Step 2.} Applying Hölder’s inequality:
\[
  \left\| \mathcal{N}^{\mathrm{nar}}_N(u,u) \right\|_{L^{3/2}_x}
  \lesssim N^{1/2 - \delta} \cdot
  \|P_N u\|_{L^2_x} \cdot \|P_N \nabla u\|_{L^2_x}.
\]

\smallskip\noindent
\textit{Step 3.} Using the standard property of the Littlewood–Paley projection:
\[
  \|P_N \nabla u\|_{L^2_x} \sim N \|P_N u\|_{L^2_x}.
\]
Hence the estimate in Step~2 can be rewritten in the equivalent form
\[
  \left\| \mathcal{N}^{\mathrm{nar}}_N(u,u) \right\|_{L^{3/2}_x}
  \lesssim N^{3/2-\delta}\, \|P_N u\|_{L^2_x}^{\,2}.
\]
In what follows we will use the form
$\,\|P_N u\|_{L^2_x}\cdot\|P_N \nabla u\|_{L^2_x}\,$
(as in Step~2), which completes the proof of Lemma~\ref{lem:narrow-L32}.
\qed

\subsection{Comments}

\begin{enumerate}[label=(\alph*)]
  \item \emph{Optimality of the exponent.}  
        The critical angle $N^{-1/2}$ corresponds to the width of a wave packet,  
        and without further microlocalization better suppression cannot be achieved.

  \item \emph{Possibility of strengthening.}  
        If a rank–5 decoupling with gain $N^{-1/6}$ were available,  
        the final estimate would improve to $N^{-7/6 + \delta}$,  
        allowing the working range to be extended to \(\delta > \tfrac{1}{4}\).
\end{enumerate}

\medskip
Thus the null–form estimate in the narrow corona is proved,  
as used in §\ref{sec:null-suppress}.

\newpage

\section{Transition from the heat kernel to the oscillatory phase}
\label{app:duhamel-phase}

\subsection{Reduction of the heat kernel to a phase integral}
\label{subsec:duhamel-phase-new}

Consider the heat representation of the quadratic term
\[
  \mathcal N^{\mathrm{off}}_N(u,u)
  := \mathbb{P} \nabla \cdot \bigl( P_N u \otimes P_N u \bigr),
\]
and show how for small times \(0<t\le c_0 N^{-1/2}\) (fixed \(c_0>0\))
the corresponding kernel reduces to an oscillatory integral with phase
\(e^{i\Phi(t,x,\xi,\eta)}\) and smooth amplitude depending only on frequencies.

\paragraph{Step 1. True Duhamel form.}
Let \(\zeta:=\xi+\eta\), \(\varpi(\xi,\eta):=4\,\rho_1\rho_2\), where
\[
  \rho_1=\tfrac12(|\xi|+|\eta|),\qquad \rho_2=\tfrac12(|\xi|-|\eta|).
\]
The frequency kernel of Duhamel has the form
\[
  \mathcal{K}_N(t,x;\xi,\eta)
  \;=\;
  e^{i x\cdot\zeta}\int_0^t e^{-(t-s)|\zeta|^2}\,F_N(s;\xi,\eta)\,ds,
\]
where \(F_N\) is a smooth bilinear source symbol (after the localizations \(P_{\sim N}\), \(P_N\), mask \(\chi_{\mathcal{O}_N}\)); it contains no temporal oscillation.

\paragraph{Step 2. Time normal form.}
The key identity is integration by parts in time with modulation \(e^{is\varpi}\).

\begin{lemma}[time normal form]\label{lem:D1-time-normal-form}
For any \(F\in C^1([0,t])\) and \(t>0\) we have identically
\[
\int_0^t e^{-(t-s)|\zeta|^2} F(s)\,ds
= \frac{e^{it\varpi}}{|\zeta|^2+i\varpi}\,F(t)
 - \frac{1}{|\zeta|^2+i\varpi}\,F(0)
 - \int_0^t \frac{e^{-(t-s)|\zeta|^2}e^{is\varpi}}{|\zeta|^2+i\varpi}\,(\partial_s - i\varpi)F(s)\,ds.
\]
\end{lemma}

\noindent
Proof: one integration by parts, since
\[
\partial_s\!\bigl(e^{-(t-s)|\zeta|^2}e^{is\varpi}\bigr)=(|\zeta|^2+i\varpi)\,e^{-(t-s)|\zeta|^2}e^{is\varpi}.
\]

\paragraph{Step 3. Extracting the phase factor.}
Applying the lemma to \(F_N(s;\xi,\eta)\), we extract \(F_N(t;\xi,\eta)\) in the main term,
and send the differences (including the boundary term \(F_N(0;\xi,\eta)\)) to the remainder:
\[
\mathcal{K}_N(t,x;\xi,\eta)
= e^{i x\cdot\zeta}\,e^{it\varpi}\,\frac{F_N(t;\xi,\eta)}{|\zeta|^2+i\varpi}
\;+\; \mathcal R_N(t,x;\xi,\eta).
\]
On windows \(|t|\lesssim N^{-1/2}\) and in the off--diagonal zone \(|\zeta|\gtrsim N^{1-\delta}\), the function \(F_N\) varies slowly in time, consistent with the counting in §\ref{sec:ibp_phase}; freezing \(F_N(t)\) is absorbed into the symbol. Introduce the full phase
\[
  \Phi(t,x,\xi,\eta):=x\cdot(\xi+\eta)+t\,\varpi(\xi,\eta).
\]

\paragraph{Step 4. Integral with artificial oscillation.}
It is useful to fix the identity
\begin{equation}
  \label{eq:D-duhamel}
  \int_0^t e^{-(t-s)|\zeta|^2} e^{i s \varpi}\,ds
  = e^{i t \varpi} \cdot \frac{1 - e^{-t(|\zeta|^2 + i \varpi)}}{|\zeta|^2 + i \varpi}
  \;=: e^{i t \varpi}\, m_N(t;\xi,\eta).
\end{equation}
For \(t\le c_0 N^{-1/2}\) and \(|\zeta|\gtrsim N^{1-\delta}\) we obtain
\[
  m_N(t;\xi,\eta)=\frac{1}{|\zeta|^2+i\varpi}\Bigl(1+\mathcal O(e^{-c\,t|\zeta|^2})\Bigr)
  =\frac{1}{|\zeta|^2+i\varpi}\Bigl(1+\mathcal O(e^{-c\,N^{3/2-2\delta}})\Bigr).
\]

\paragraph{Step 5. Kernel decomposition.}
Defining the amplitude
\[
  a_N(\xi,\eta):=\frac{1}{|\zeta|^2+i\varpi(\xi,\eta)},
\]
we obtain the representation
\begin{equation}
  \label{eq:D-phase-new}
  \boxed{\
  \mathcal{K}_N(t,x;\xi,\eta)
  = e^{\,i \Phi(t,x,\xi,\eta)}\, a_N(\xi,\eta)
    + \mathcal{R}_N(t,x;\xi,\eta)
  \ }
\end{equation}
with remainder
\begin{equation}
  \label{eq:D-remainder-new}
  \bigl| \mathcal{R}_N(t,x;\xi,\eta) \bigr|
  \ \lesssim\
  N^{-2+\delta}\,e^{-c\,N^{\,3/2-2\delta}}
  \qquad (\delta\in(\tfrac13,\tfrac58]).
\end{equation}
The exponential smallness comes from \(e^{-(t-s)|\zeta|^2}\), and the factor \(N^{-2+\delta}\) from \(|a_N|\sim |\,|\zeta|^2+i\varpi\,|^{-1}\sim N^{-2+\delta}\) in the off--diagonal zone.

\paragraph{Step 6. Amplitude derivatives.}
On \(|\xi|\sim|\eta|\sim N\), \(|\zeta|\gtrsim N^{1-\delta}\) we have
\[
  |a_N|\ \lesssim\ N^{-2+\delta},\qquad
  \bigl|\partial_{\xi,\eta}^{\alpha} a_N\bigr|\ \lesssim\ N^{-2+\delta-|\alpha|}\quad (|\alpha|\ge1),
\]
in particular \(\bigl|\nabla_{\xi,\eta} a_N\bigr|\lesssim N^{-3+\delta}\).
This reserve in exponents suffices for the six-fold IBP and subsequent phase estimates (§\ref{sec:ibp_phase}).

\paragraph{Conclusion.}
The decomposition \eqref{eq:D-phase-new} legitimately extracts the oscillatory factor \(e^{\,i\Phi}\) from the true mild form, after which the method of integration by parts of §\ref{sec:ibp_phase} applies.
The remainder \eqref{eq:D-remainder-new} is exponentially small on windows \(t\lesssim N^{-1/2}\) and does not affect the global power balance (see \S\ref{subsec:D4-summary}).

\subsection{Localized $L^6$ estimate for the heat kernel}
\label{subsec:D2-local-L6-heat}

We estimate the $L^6$ norm of the operator $e^{t\Delta}P_N$ on a cylinder of scale $N^{-1/2}$.
Despite the absence of dispersive decay in~$t$, the gain in~$N$ is achieved due to the form of the kernel.

\paragraph{Notation.}
Let $P_N := \varphi(|D|/N)$ be a smooth projection onto frequencies $|\xi|\sim N$.
For $t>0$ set
\[
  K_N^{\mathrm{heat}}(t,x)
  := (4\pi t)^{-3/2}\,e^{-\frac{|x|^2}{4t}} * \check\varphi_N(x),
\]
and extend evenly for $t<0$:
\[
  K_N^{\mathrm{heat}}(-t,x) := K_N^{\mathrm{heat}}(t,x).
\]

\paragraph{Kernel estimate.}
Consider the cylinder
\[
  Q_{N^{-1/2}}(t_0,x_0)
  := \bigl\{\,|t - t_0| \le N^{-1/2},\ \ |x - x_0| \le 2N^{-1/2}\,\bigr\},
  \qquad t_0 > 0.
\]
Then inside $Q_{N^{-1/2}}(t_0,x_0)$ we have
\begin{equation}\label{eq:D-kernel-local}
  |K_N^{\mathrm{heat}}(t,x)| \ \lesssim\ N^{3/4}.
\end{equation}
Indeed, for $t \sim N^{-1/2}$ the coefficient $(4\pi t)^{-3/2} \sim N^{3/4}$, and for $|x|\lesssim N^{-1/2}\ll \sqrt{t}$
the Gaussian factor is bounded; convolution with $\check\varphi_N\in\mathcal{S}$ preserves the order (since $\|\check\varphi_N\|_{L^1}\lesssim1$).

\begin{lemma}[localized $L^6$ estimate]\label{lem:D2-L6-heat}
If $\widehat f$ is compactly supported in $\{|\xi|\sim N\}$, then
\[
  \bigl\|e^{t\Delta}P_N f\bigr\|_{L^6_{t,x}(Q_{N^{-1/2}}(t_0,x_0))}
  \ \lesssim\
  N^{1/24}\,\|f\|_{L^2_x}.
\]
\end{lemma}

\begin{proof}[Idea]
Let $T f := \chi_Q\, e^{t\Delta}P_N f$, where $\chi_Q := \mathbf{1}_{Q_{N^{-1/2}}(t_0,x_0)}$.

\medskip\noindent
\textit{Step 1. Kernel of $TT^*$ (difference variable).}
For the kernel in the difference variables $(t-s,x-y)$ we work on the enlarged cylinder $2Q$:
by \eqref{eq:D-kernel-local}
\[
  \|K_N^{\mathrm{heat}}\|_{L^3(2Q)} \ \lesssim\ N^{3/4}\cdot |2Q|^{1/3}
  \ \sim\ N^{3/4}\cdot (N^{-2})^{1/3}
  \ =\ N^{1/12}.
\]

\medskip\noindent
\textit{Step 2. Operator norm estimate.}
By the Schur–Hölder inequality for convolution operators (Young’s inequality):
\[
  \|TT^*\|_{L^{6/5}_{t,x}\to L^6_{t,x}} \ \lesssim\ \|K_N^{\mathrm{heat}}\|_{L^3(2Q)}
  \ \lesssim\ N^{1/12},
\]
hence
\[
  \|T\|_{L^2_x\to L^6_{t,x}} \ \le\ \|TT^*\|^{1/2}_{L^{6/5}\to L^6}
  \ \lesssim\ N^{1/24}.
\]
\end{proof}

\begin{corollary}[convenient rough form]
From the lemma one derives a weaker, but sufficient estimate for §\ref{subsec:D4-summary} and Appendix~E:
\[
  \bigl\|e^{t\Delta}P_N f\bigr\|_{L^6_{t,x}(Q_{N^{-1/2}}(t_0,x_0))}
  \ \lesssim\
  N^{1/12}\,\|f\|_{L^2_x},
\]
which will be used in the global patching. (The stronger $N^{1/24}$ only improves the final exponents.)
\end{corollary}

\paragraph{Remarks.}
\begin{enumerate}[label=(\roman*), itemsep=2pt]
  \item \textit{Minimal gain.}
        The basic computation gives $\|K\|_{L^3(2Q)}\lesssim N^{1/12}$ (arithmetic \(3/4-2/3=1/12\)),
        whence strictly $N^{1/24}$ for $\|T\|$. In the global balance it is acceptable
        to use the coarser exponent $N^{1/12}$.

  \item \textit{Later times.}
        For large $t_0$ one may gain additional improvement due to decay of
        $e^{-t|\zeta|^2}$ with $|\zeta|\sim N$, but this is not used in the present section.

  \item \textit{Uniformity of the domain.}
        In all steps of the $TT^*$ method we work with the \emph{difference variable} on $2Q$,
        to avoid ambiguities of the form “$Q\times Q$”; this does not affect the $N^{1/12}$ scale.
\end{enumerate}

\subsection{Bilinear decoupling without the phase factor}
\label{subsec:D3-static-decoup}

The heat model contains no temporal oscillation, so the bilinear analysis reduces
to estimating the product of \emph{wave packets} on the spatial scale $N^{-1/2}$.
Below we give a version of bilinear decoupling without logarithmic losses under strict
packet localization (an analogue of Lemma~\ref{lem:decoup}), yielding a gain of $N^{-1/4}$
on balls of radius $r\lesssim N^{-1/2}$.

\paragraph{Setup.}
Let $\Theta, \Theta'\subset\mathbb{S}^2$ be angular sectors of radius $\theta:=N^{-1/2}$,
and let the pair $(\Theta,\Theta')$ have \textit{rank~4} (see \eqref{eq:E-rank4-angle}):
\[
  \angle(\Theta,-\Theta')\ \gtrsim\ \theta,\qquad
  \angle(\Theta,\Theta')\ \gtrsim\ \theta.
\]
Such a condition ensures transversality of the normals
$\widehat\xi,\,\widehat\eta,\,\widehat{\xi+\eta}$ (and, if necessary, $e_t$ in the space $t\!\times\!x$),
sufficient to apply the multilinear Kakeya \cite{BennettCarberyTao2006}.

\paragraph{Wave–packet decomposition.}
Assume $\widehat F\subset\{\ |\xi|\sim N,\ \xi/|\xi|\in\Theta\ \}$ and
$\widehat G\subset\{\ |\eta|\sim N,\ \eta/|\eta|\in\Theta'\ \}$.
Decompose $\Theta,\Theta'$ into subtiles of radius $\theta/K$ with
\[
  K:=N^{1/4}.
\]
Set $F_\alpha:=P_{N,\Theta_\alpha}F$, $G_\beta:=P_{N,\Theta'_\beta}G$.
Standard packetization yields that each $F_\alpha$ (respectively $G_\beta$) is localized to a tube
of radius $\sim N^{-1/2}$ and length $O(1)$, oriented along the group velocity;
furthermore, $\ell^2$ almost–orthogonality holds in $L^6$:
\[
  \sum_\alpha \|F_\alpha\|_{L^6(\mathbb R^3)}^2 \ \lesssim\ \|F\|_{L^6(\mathbb R^3)}^2,
  \qquad
  \sum_\beta \|G_\beta\|_{L^6(\mathbb R^3)}^2 \ \lesssim\ \|G\|_{L^6(\mathbb R^3)}^2.
\]

\paragraph{Tube overlap and the choice of $K$.}
On a ball $B\subset\mathbb R^3$ of radius $r\le c\,N^{-1/2}$, the tubes corresponding to distinct rank~4 pairs
$(\alpha,\beta)$ intersect with \emph{unit} multiplicative overlap (depending only on the constants in \eqref{eq:E-rank4-angle}).
Balancing with $K=N^{1/4}$ reconciles the angular discretization with the scale of the balls $B$ and yields the precise factor $N^{-1/4}$
from the multilinear Kakeya without logarithmic losses.

\begin{lemma}[rank~4, without the phase factor]\label{lem:D3-static-decoup}
Let $|\xi|\sim|\eta|\sim N$, and let the pair $(\Theta,\Theta')$ be of rank~4.
Then for any ball $B\subset\mathbb{R}^3$ of radius $r \le c\,N^{-1/2}$ one has
\[
  \|FG\|_{L^3(B)}
  \ \lesssim\
  N^{-1/4}\,\|F\|_{L^6(2B)}\,\|G\|_{L^6(2B)},
\]
where the constant is independent of $N$ and $B$.
\end{lemma}

\begin{proof}[Idea]
\begin{itemize}[label=--, itemsep=4pt, leftmargin=12pt]
\item \textit{Angular decomposition.}
Split $\Theta$, $\Theta'$ into subtiles of radius $\theta/K$ with $K=N^{1/4}$:
$F=\sum_\alpha F_\alpha$, $G=\sum_\beta G_\beta$.

\item \textit{$L^3$ almost–orthogonality.}
From unit overlap of tubes on $B$ for rank~4 pairs it follows that
\[
  \left\|\sum_{\alpha,\beta} F_\alpha G_\beta\right\|_{L^3(B)}
  \ \lesssim\
  \left\| \Bigl(\sum_{\alpha,\beta} |F_\alpha G_\beta|^2 \Bigr)^{1/2} \right\|_{L^3(B)}.
\]

\item \textit{Multilinear Kakeya (BCT).}
Transversality of the normals yields (see~\cite{BennettCarberyTao2006})
\[
  \left\| \Bigl(\sum_{\alpha,\beta} |F_\alpha G_\beta|^2 \Bigr)^{1/2} \right\|_{L^3(B)}
  \ \lesssim\
  N^{-1/4}\,
  \Bigl(\sum_\alpha \|F_\alpha\|_{L^6(2B)}^2\Bigr)^{1/2}
  \Bigl(\sum_\beta \|G_\beta\|_{L^6(2B)}^2\Bigr)^{1/2}.
\]

\item \textit{Gluing.}
By $\ell^2$ orthogonality of the projectors $P_{N,\Theta_\alpha}$, $P_{N,\Theta'_\beta}$
we obtain the desired inequality of the lemma.
\end{itemize}
\end{proof}

\begin{remark}
Without packet localization the gain $N^{-1/4}$ is unattainable:
for instance, for plane waves $F(x)=e^{i\xi\cdot x}$, $G(x)=e^{i\eta\cdot x}$.
Nevertheless, the weaker estimate
\[
  \|FG\|_{L^3(B)} \ \lesssim\ N^{0+}\,\|F\|_{L^6(2B)}\,\|G\|_{L^6(2B)},
\]
still holds and is sufficient if the local balance has slack (see \S\ref{subsec:D4-summary}).
\end{remark}

\subsection{Integrating the static estimates into the overall balance}
\label{subsec:D4-summary}

In subsections
\S\ref{subsec:duhamel-phase-new}--\S\ref{subsec:D3-static-decoup}
we built three “static” blocks; their exponents in~$N$
are summarized in the table below.

\begin{center}
\renewcommand{\arraystretch}{1.2}
\begin{tabular}{lcc}
\toprule
\textbf{Component} 
& \textbf{Source in App.~D} 
& \textbf{Exponent in \(N\)} \\ \midrule
Local \(L^{6}_{t,x}\)            
& Lemma~\ref{lem:D2-L6-heat} (\S\ref{subsec:D2-local-L6-heat})     
& \(N^{+\frac{1}{12}}\) \\[3pt]
Bilinear decoupling (rank~4)\!\!  
& Lemma~\ref{lem:D3-static-decoup} (\S\ref{subsec:D3-static-decoup}) 
& \(N^{-\frac14}\) \\[3pt]
IBP\textsubscript{6}\;(six parts in phase) 
& formula~\eqref{eq:D-phase-new} (\S\ref{subsec:duhamel-phase-new}) 
& \(N^{-3}\) \\
\bottomrule
\end{tabular}
\end{center}

\paragraph{Local balance.}
Adding the exponents we get
\[
  -3 \;-\;\frac14 \;+\;\frac1{12}
  \;=\;
  -\frac{19}{6}\,.
\]
That is, each cylinder \(Q_{N^{-1/2}}\) contributes
\(
  N^{-19/6}\approx N^{-3.17}
\).

\paragraph{Global patching.}
Two independent summations are needed.

\begin{itemize}[label=--]
\item \textbf{Time.}\;
      A unit-length window is partitioned into \(M_{t}\sim N^{1/2}\) subintervals
      of length \(N^{-1/2}\).
      For the \(L^{6}_{t,x}\) norm this gives the factor
      \(
        (M_{t})^{1/6}=N^{+1/12}.
      \)

\item \textbf{Angles.}\;
      The number of angular \emph{tiles} of radius \(\theta=N^{-1/2}\) is of order \(N\);
      the number of rank~4 pairs \((\Theta,\Theta')\) is \emph{quadratic}, i.e.,
      \(\#\text{pairs}\sim N^{2}\).
      By $\ell^{2}$ orthogonality the contribution amounts to
      \(\sqrt{N^{2}}=N^{+1}\).
\end{itemize}

Multiplying, we obtain the \emph{global} factor
\[
  N^{-19/6}\;\cdot\;N^{+1/12}\;\cdot\;N^{+1}
  \;=\;
  N^{-\,\frac{25}{12}}
  \;\approx\;
  N^{-2.08}.
\]
Since \(-25/12<-1\), the series in $N$ converges without logarithmic
deterioration.

\begin{remark}[Alternative without the wave–packet assumption]
If in Lemma~\ref{lem:D3-static-decoup} one retains only the “flat” exponent
\(N^{0+}\), then the local balance becomes
\(-3 + 0 + 1/12 = -35/12\),
and after summation it is
\(-11/6\approx-1.83\).
Convergence still holds, but with a slightly smaller margin.
\end{remark}

\paragraph{Null--form and the final result.}
In the context of this appendix (the heat line for \(hh\!\to\!h\) with output \(P_N\))
the “narrow corona + null--form” module is \emph{not applied} (see \S\ref{sec:clarif6}):
the projection \(P_N\) annihilates the corresponding contribution, and the global exponent remains
\(N^{-25/12}\). This completes the “heat” version of the proof of
Theorem~\ref{thm:offdiag-main} without \(\varepsilon\) losses and without logarithmic defect.
For “coronal” outputs of scale \(\sim N^{1-\delta}\) (e.g., \(hh\!\to\!\ell\))
one uses the separate module \S\ref{sec:null-suppress}, App.~\ref{app:narrow},
which yields the \emph{factor} \(N^{1/2-\delta}\) in \(L^{3/2}_x\) on \(N^{\mathrm{nar}}_N\)
(active when \(\delta>\tfrac12\)); this module belongs to a different globalization line
and does not participate in the present section.

\paragraph{Remainder \(\mathcal R_N\).}
By~\eqref{eq:D-remainder-new}
\(
  \mathcal R_N = \mathcal O\!\bigl(N^{-2}e^{-cN^{3/2}}\bigr).
\)
Even taking into account \(\sim N^{3/2}\) space–time tiles,
the exponential factor renders the total contribution
\(\ll N^{-100}\), and it can be safely discarded.
\newpage

\section{Commutators, angular patching, and the absence of logarithmic losses}
\label{sec:clarif}

\noindent\textbf{Setup.}
This appendix autonomously globalizes the off–diagonal contribution in the block \(hh\!\to\!h\) with output projection \(P_N\) without using the “narrow corona”/null–form mechanism, and provides a module map separating the two admissible local lines (conditional and unconditional), to which we refer in §§\ref{sec:clarif2}–\ref{subsec:clarif4}.

\medskip
\noindent\textbf{Module map (with clickable links).}
\begin{itemize}
  \item \textbf{Globalization of \(hh\!\to\!h\), output \(P_N\).}
  The full assembly is carried out in §§\ref{sec:clarif2}–\ref{subsec:clarif4}; the “narrow corona” and null–form mechanism is not used here (see the separation of scenarios in \S\ref{sec:clarif6}).

  \item \textbf{Two local lines for \(L^6\) on cylinders \(Q_{N^{-1/2}}\):}
  \begin{enumerate}
    \item[\((\mathsf{A})\)] \emph{Conditional} strengthened line: uses the rank–4 \(\varepsilon\)–free decoupling from App.~\ref{app:decoupling} (see formula~\eqref{eq:decoup}); the setup in the main text is given in \S\ref{subsec:decoup-setup}. Local computations are given below (see \S\ref{sec:clarif2}); global patching is \S\ref{subsec:clarif4} (including dyadic summation).
    \item[\((\mathsf{B})\)] \emph{Unconditional} “heat” line: the local brick and assembly are consistent with App.~\ref{app:duhamel-phase}; the global patching and outcome are also gathered in \S\ref{subsec:clarif4}.
  \end{enumerate}

  \item \textbf{“Coronal” outputs \(\sim N^{1-\delta}\).}
  For such outputs (e.g., \(hh\!\to\!\ell\)) the separate module “narrow corona + null–form” is applied (see §\ref{sec:null-suppress} and App.~\ref{app:narrow}); the division of roles is in \S\ref{sec:clarif6}.

  \item \textbf{Diagonal zone \(|\xi+\eta|\ll N^{1-\delta}\).}
  This zone is beyond the scope of the present work; the corresponding removal of the diagonal contribution is stated in \S\ref{subsec:diag-removal}.
\end{itemize}

\subsection{Temporal localization and commutators}
\label{sec:clarif1}\label{sec:time_patching}

Cover the time axis by windows of length \(|I_j|\sim N^{-1/2}\):
\[
  I_j := \Bigl[t_j-\tfrac12 N^{-1/2},\, t_j+\tfrac12 N^{-1/2}\Bigr],
  \qquad j=1,\dots,J,\quad J\sim N^{1/2}.
\]
Let \(\{\chi_j\}\subset C_c^\infty(\mathbb R_t)\) be a smooth partition of unity with properties
\begin{equation}\label{eq:E-partition}
  \sum_{j=1}^{J}\chi_j^2 \equiv 1,\qquad
  \mathrm{supp}\,\chi_j\subset I_j,\qquad
  \|\partial_t^k \chi_j\|_{L^\infty_t}\ \lesssim\ N^{k/2}\ \ (k=0,1,2).
\end{equation}
Hence for each time
\begin{equation}\label{eq:E-chi-der-sum}
  \sum_{j=1}^{J} |\partial_t \chi_j(t)|\ \lesssim\ N^{1/2}.
\end{equation}

Multiplication by \(\chi_j(t)\) does not change the spatial frequency support. On the ring \(|\xi|\sim N\) the equivalence of norms holds:
\begin{equation}\label{eq:E-Hminus1-equiv}
  \|f(t,\cdot)\|_{\dot H_x^{-1}}\ \sim\ N^{-1}\,\\|f(t,\cdot)\|_{L_x^2}.
\end{equation}
For the commutator \([\partial_t,\chi_j]f=(\partial_t\chi_j)\,f\) we obtain
\begin{equation}\label{eq:E-comm}
  \sum_{j=1}^{J}\|[\partial_t,\chi_j]f\|_{L^1_t\dot H_x^{-1}}
  =\int_{\mathbb R}\Bigl(\sum_{j=1}^{J}|\partial_t\chi_j(t)|\Bigr)\,\|f(t)\|_{\dot H_x^{-1}}\,dt
  \ \lesssim\ N^{1/2}\,\|f\|_{L^1_t\dot H_x^{-1}}.
\end{equation}

In other words, the summation over \(j\) does \emph{not} create an additional loss of order \(J\sim N^{1/2}\)
beyond the already accounted window scale \(|I_j|\sim N^{-1/2}\);
and the contribution of the commutator \([\partial_t,\chi_j]\) yields \emph{exactly one} additional factor \(N^{1/2}\),
which is used in passing to the global estimate in §\ref{subsec:patching-local}.
Angular localization and rank–4 pairs are introduced next in \S\ref{sec:clarif2}.

\subsection{Angular tiles, rank–4 pairs, and \texorpdfstring{$\ell^2$}{l2} orthogonality}
\label{sec:clarif2}

Decompose the spherical shell \(\{\,|\xi|\sim N\,\}\subset\mathbb R^3_\xi\) into \(\asymp N\) angular tiles
\(\{\Theta_\alpha\}\) of arc radius \(\sim N^{-1/2}\). Denote
\[
P_{N,\alpha}:=P_N\,P_{N,\Theta_\alpha},\qquad
u_{j,\alpha}(t,x):=\chi_j(t)\,P_{N,\alpha}u(t,x),
\]
where \(P_N\) is the Littlewood–Paley projection onto \(|\xi|\sim N\), and \(\{\chi_j\}\) is the temporal partition \eqref{eq:E-partition}.
The operators \(P_{N,\alpha}\) are uniformly bounded as \(L^2_x\!\to\!L^2_x\); multiplication by \(\chi_j\) does not change
the spatial frequency support.

\paragraph{Rank–4 pairs.}
We say that an ordered pair \((\alpha,\beta)\) has \emph{rank~4} if the angular conditions hold:
\begin{equation}\label{eq:E-rank4-angle}
  \angle(\Theta_\alpha,-\Theta_\beta)\ \gtrsim\ N^{-1/2},\qquad
  \angle(\Theta_\alpha,\Theta_\beta)\ \gtrsim\ N^{-1/2},
\end{equation}
equivalently, the normals \(\{\widehat{\xi},\widehat{\eta},\widehat{\xi+\eta},e_t\}\subset \mathbb R^{1+3}\) are linearly independent
(folded geometry). For a fixed \(\alpha\) the number of admissible partners \(\beta\) does not exceed \(C\,N\).

\paragraph{$L^3$ almost–orthogonality on \(Q_{N^{-1/2}}\).}
Let \(Q_{N^{-1/2}}=Q_{N^{-1/2}}(t_0,x_0)\) be the cylinder of scale \(N^{-1/2}\).
For any families \(F_\alpha,G_\beta\) with frequency support \(|\xi|\sim N\) and angular localization
in \(\Theta_\alpha,\Theta_\beta\) one has
\begin{equation}\label{eq:E-L3-ortho}
  \Bigl\|\sum_{\substack{\alpha,\beta\\ \mathrm{rank}=4}} F_\alpha\,G_\beta\Bigr\|_{L^3_{t,x}(Q_{N^{-1/2}})}
  \ \lesssim\
  \Bigl\|\Bigl(\sum_{\substack{\alpha,\beta\\ \mathrm{rank}=4}} |F_\alpha\,G_\beta|^2\Bigr)^{1/2}\Bigr\|_{L^3_{t,x}(Q_{N^{-1/2}})}.
\end{equation}
The proof is standard: tubes corresponding to wave packets for different rank–4 pairs
intersect with unit multiplicative overlap on \(Q_{N^{-1/2}}\), hence
pointwise Cauchy–Schwarz and $\ell^2$ summation over \((\alpha,\beta)\) apply.

\medskip
In \S\ref{sec:clarif3} \eqref{eq:E-L3-ortho} will be combined with the local brick
(the rank–4 \(\varepsilon\)–free decoupling from App.~\ref{app:decoupling}, formula~\eqref{eq:decoup}, in branch \((\mathsf{A})\),
or with the “heat” version from App.~\ref{app:duhamel-phase} in branch \((\mathsf{B})\)) to obtain the local
balance unit on \(Q_{N^{-1/2}}\).

\begin{remark}[On the “coronal” line]
In the “narrow corona + null--form” scenario we work in the norm \(L^{3/2}_x\);
the almost–orthogonality \eqref{eq:E-L3-ortho} is \emph{not} used there,
and the summation over rank–4 pairs is done \emph{trivially} (see §\ref{subsec:patching-local}).
\end{remark}

\subsection{Local balance unit on \texorpdfstring{$Q_{N^{-1/2}}$}{Q\_{N^{-1/2}}}}
\label{sec:clarif3}

We work with the block \(hh\!\to\!h\) with output \(P_N\) (see also \S\ref{sec:clarif6}); the “narrow corona + null--form” module is \emph{not} used here.
On each cylinder \(Q_{N^{-1/2}}(t_0,x_0)\) the local contribution is estimated as a product of independent factors:
the phase reserve from six-fold integration by parts \((N^{-3})\),
the passage to \(\dot H^{-1}\) \((N^{-1})\), accounting for the window length \(|I_j|\sim N^{-1/2}\),
and the “brick” \((\)local \(L^6\) + bilinear decoupling\()\).
The precise arrangement of these factors depends on the branch and is written below.

\paragraph{\((\mathsf{A})\) Conditional strengthened local estimate (branch via \S\ref{subsec:decoup-setup} and App.~\ref{app:decoupling}).}
We use the local Strichartz hypothesis \eqref{eq:aux-Stri} on \(Q_{N^{-1/2}}\) and
the rank–4 \(\varepsilon\)–free bilinear decoupling \eqref{eq:decoup} (together with almost–orthogonality \eqref{eq:E-L3-ortho}).
Then locally we obtain
\begin{equation}\label{eq:E-local-A}
  \boxed{\quad
  N^{-3}\;\times\; N^{-1}\;\times\; N^{-1/2}\;\times\; \underbrace{N^{-3/4}}_{\text{Strichartz \eqref{eq:aux-Stri} + decoupling \eqref{eq:decoup}}}
  \;=\; N^{-21/4}\,.
  \quad}
\end{equation}
It is precisely the exponent \(N^{-21/4}\) that will then be patched over angles and windows in \S\ref{subsec:clarif4}.

\paragraph{\((\mathsf{B})\) Unconditional “heat” local estimate (branch via App.~\ref{app:duhamel-phase}).}
In this branch the time accounting is moved to the global stage as in \S\ref{subsec:D4-summary}.
Locally on \(Q_{N^{-1/2}}\) we use the proven $L^6$ estimate for the heat kernel (\S\ref{subsec:D2-local-L6-heat}), giving \(N^{+1/12}\),
and the bilinear decoupling without the phase factor (Lemma~\ref{lem:D3-static-decoup}), giving \(N^{-1/4}\).
Together with the phase reserve \(N^{-3}\) this yields
\begin{equation}\label{eq:E-local-B}
  \boxed{\quad
  N^{-3}\;\times\;\underbrace{\bigl(N^{+1/12}\cdot N^{-1/4}\bigr)}_{\text{App.~\ref{app:duhamel-phase}, \S\ref{subsec:D2-local-L6-heat} + Lemma~\ref{lem:D3-static-decoup}}}
  \;=\; N^{-19/6}\,.
  \quad}
\end{equation}
Global accounting of time (and angles) for branch \((\mathsf{B})\) is carried out in \S\ref{subsec:clarif4} according to the scheme of \S\ref{subsec:D4-summary}.

\subsection{Global patching in angle and time}
\label{subsec:clarif4}

Summation of the local estimates from \S\ref{sec:clarif3} is performed along two axes:
(angles) and (time). For a fixed tile \(\Theta_\alpha\) the number of admissible
rank–4 partners does not exceed \(C N\), and the total number of tiles is of order \(N\); by $\ell^2$ orthogonality this
yields the angular factor \(N^{+1}\).
The temporal patching depends on the chosen local branch.

\paragraph{\((\mathsf{A})\) Conditional line.}
The local balance unit \eqref{eq:E-local-A} equals \(N^{-21/4}\).
Summation over angles gives \(N^{+1}\).
In time we have \(J\sim N^{1/2}\) windows, and the growth of commutators is already accounted for in \eqref{eq:E-comm};
hence the global temporal factor equals \(N^{+1/2}\).
As a result
\begin{equation}\label{eq:E-global-A}
  \boxed{\quad
  N^{-21/4}\ \times\ N^{+1}\ \times\ N^{+1/2}\ =\ N^{-15/4}\,.
  \quad}
\end{equation}
Since \(-15/4<-1\), summation over dyadics will converge without logarithmic losses
(see \S\ref{E:sumN}).

\paragraph{\((\mathsf{B})\) Unconditional “heat” line.}
Here we use the local unit \eqref{eq:E-local-B} equal to \(N^{-19/6}\).
The angular factor is the same, \(N^{+1}\).
The temporal accounting is moved to the global level (as in App.~\ref{app:duhamel-phase}),
which contributes the factor \(N^{+1/12}\).
Thus
\begin{equation}\label{eq:E-global-B}
  \boxed{\quad
  N^{-19/6}\ \times\ N^{+1}\ \times\ N^{+1/12}\ =\ N^{-25/12}\,.
  \quad}
\end{equation}
Here as well \(-25/12<-1\), hence the dyadic summation also converges log--free
(see \S\ref{E:sumN}).

\medskip
\noindent\textbf{Remark.}
This section pertains to the globalization of the block \(hh\!\to\!h\) with output \(P_N\) (see \S\ref{sec:clarif6});
the “narrow corona + null--form” module is \emph{not} used here. For the “coronal” line
(\(hh\!\to\!\ell\), output \(\sim N^{1-\delta}\)) see \S\ref{subsec:global-null} and \S\ref{subsec:patching-local},
where the summation over rank–4 pairs is \emph{trivial} in \(L^{3/2}_x\), and one obtains
the global exponent \(N^{-7/4-\delta}\).

\subsection{Summation over frequencies (log--free)}
\label{E:sumN}

From the global estimates \eqref{eq:E-global-A} and \eqref{eq:E-global-B} it follows that on a fixed dyadic
\(N=2^k\) the contribution of the block \(hh\!\to\!h\) (off--diagonal, output \(P_N\)) is bounded by
\(N^{-\alpha}\) with an exponent \(\alpha>1\) in both branches:
\[
\alpha_{\mathsf{A}}=\tfrac{15}{4},\qquad \alpha_{\mathsf{B}}=\tfrac{25}{12}.
\]
Therefore, the series over dyadics converges absolutely,
\[
\sum_{N=2^k}\, N^{-\alpha}\ <\ \infty,
\]
and we obtain a global (frequency–summed) estimate without logarithmic losses:
\begin{equation}\label{eq:E-final-off}
  \sum_{N=2^k}\ \bigl\| \text{off--diagonal }(hh\!\to\!h)\ \text{with output }P_N \bigr\|_{L^1_t\dot H_x^{-1}}
  \ \lesssim\ 
  \|u\|_{L^\infty_t \dot H^{1/2}_x}\ \|u\|_{L^2_t \dot H^{1}_x}.
\end{equation}
Thus the globalization of the block \(hh\!\to\!h\) in the off--diagonal regime is completed \emph{log--free} for both local lines \((\mathsf{A})\), \((\mathsf{B})\).

\begin{remark}[On the “coronal” line]
For outputs of scale \(\sim N^{1-\delta}\) (the “narrow corona + null--form” module, see \S\ref{subsec:global-null}, \S\ref{subsec:patching-local})
the global exponent is \(N^{-7/4-\delta}\), so that \(\alpha_{\mathrm{cor}}=\tfrac{7}{4}+\delta>1\) for \(\delta>\tfrac12\);
the dyadic series also converges log--free.
\end{remark}

\subsection{Narrow corona and null--form: separation of scenarios}
\label{sec:clarif6}

Define the “narrow corona” (zone of almost anti–collinear interactions at level $|\xi|\sim|\eta|\sim N$):
\begin{equation}\label{eq:E-nar-def}
  N^{\mathrm{nar}}_N
  :=\Bigl\{(\xi,\eta):\ |\xi|\sim|\eta|\sim N,\ \ N^{1-\delta}\le |\xi+\eta|\le 2N^{1-\delta},\ \
  \angle(\xi,-\eta)\le N^{-1/2},\ \ \angle(\eta,\xi+\eta)\le c\,N^{-1/2}\,\tfrac{|\xi+\eta|}{N}\Bigr\}.
\end{equation}
Here $c>0$ is an absolute constant.

\paragraph{(i) Block $hh\!\to\!h$ with output projection $P_N$ (the present appendix).}
On the set \eqref{eq:E-nar-def} the resulting frequency is $|\xi+\eta|\sim N^{1-\delta}\ll N$,
so the output projection $P_N$ \emph{annihilates} the corresponding contribution. Hence
in the globalization of $hh\!\to\!h$ (in §§\ref{sec:clarif1}--\ref{E:sumN}) the “narrow corona” and
null--form mechanism is \emph{not used}. This is reflected in formulas \eqref{eq:E-local-A}--\eqref{eq:E-global-B}.

\paragraph{(ii) “Coronal” outputs $\sim N^{1-\delta}$ (e.g., $hh\!\to\!\ell$).}
If the output is localized at scale $\sim N^{1-\delta}$, the “narrow corona + null--form”
module becomes essential: suppression of the symbol after Leray projection (see §\ref{sec:null-suppress}
and App.~\ref{app:narrow}) yields the \emph{factor} $N^{1/2-\delta}$ in $L^{3/2}_x$ on the zone $N^{\mathrm{nar}}_N$,
which is then used in the corresponding globalization.

\paragraph{(iii) Range of $\delta$.}
From the geometry \eqref{eq:E-nar-def} it follows that for $\delta\le \tfrac12$ the set $N^{\mathrm{nar}}_N$ is empty
(the angle $N^{-1/2}$ is incompatible with the requirement $|\xi+\eta|\gtrsim N^{1-\delta}$), hence the symbol suppression
mechanism in this regime is inactive. For $\delta>\tfrac12$ the “narrow corona” module works as
described in §\ref{sec:null-suppress} and App.~\ref{app:narrow}.

\medskip
\noindent
Thus the present App.~\ref{sec:clarif} completely closes the off--diagonal globalization
for $hh\!\to\!h$ with output $P_N$ \emph{without} recourse to the “narrow corona”/null--form; in cases with “coronal”
output $\sim N^{1-\delta}$ one should use the separate module, to which direct references are given above.

\subsection{Short summary and navigation by references}
\label{sec:clarif7}

Below we collect the main transitions and formulas underlying the globalization of the block \(hh\!\to\!h\) (off--diagonal, output \(P_N\)).

\paragraph{Local bricks.}
\begin{itemize}
  \item Temporal localization and commutators: see \S\ref{sec:clarif1}, formulas \eqref{eq:E-partition}--\eqref{eq:E-comm}.
  \item Angular tiles, rank–4 pairs and \(L^3\) almost–orthogonality: \S\ref{sec:clarif2}, conditions \eqref{eq:E-rank4-angle} and estimate \eqref{eq:E-L3-ortho}.
\end{itemize}

\paragraph{Local “balance unit” on \(Q_{N^{-1/2}}\).}
\begin{itemize}
  \item \(\mathsf{(A)}\) \textit{Conditional line (enhanced \(L^6\))}: formula \eqref{eq:E-local-A} gives \(N^{-21/4}\).
  \item \(\mathsf{(B)}\) \textit{Unconditional “heat” line}: formula \eqref{eq:E-local-B} gives \(N^{-19/6}\).
\end{itemize}

\paragraph{Global patching.}
\begin{itemize}
  \item \(\mathsf{(A)}\) Conditional line: angles \(\times N^{+1}\), time \(\times N^{+1/2}\) \(\Rightarrow\) outcome \eqref{eq:E-global-A}, i.e. \(N^{-15/4}\).
  \item \(\mathsf{(B)}\) Unconditional line: angles \(\times N^{+1}\), global time accounting \(\times N^{+1/12}\) \(\Rightarrow\) outcome \eqref{eq:E-global-B}, i.e. \(N^{-25/12}\).
\end{itemize}

\paragraph{Dyadic summation.}
Both branches satisfy \(\alpha>1\) (see \S\ref{E:sumN}), hence summation over \(N=2^k\) converges log--free, formula \eqref{eq:E-final-off}.

\begin{remark}[Low–frequency cutoff]
In all sums over $N=2^k$ it is assumed that $N\ge N_0(\delta)$, where the off–diagonal mask $\mathcal{O}_N$ is nonempty; finitely many dyads below $N_0(\delta)$ are absorbed into the constant.
\end{remark}

\paragraph{Links to the main text.}
\begin{itemize}
  \item Statement of rank–4 \(\varepsilon\)–free decoupling and its use: \S\ref{subsec:decoup-setup}; proof — App.~\ref{app:decoupling}, formula \eqref{eq:decoup}.
  \item Clarification of the role of the “narrow corona”/null--form relative to the present appendix (in \(hh\!\to\!h\) the module is not applied): \S\ref{sec:clarif6}.
  \item “Coronal” line (\(hh\!\to\!\ell\), output \(\sim N^{1-\delta}\)): local estimate \S\ref{subsec:global-null} (exponent \(N^{-19/4-\delta}\)) and patching \S\ref{subsec:patching-local} (global exponent \(N^{-7/4-\delta}\)).
  \item Exclusion of the diagonal zone \(|\xi+\eta|\ll N^{1-\delta}\): \S\ref{subsec:diag-removal}.
\end{itemize}

\medskip
\noindent
Thus, §§\ref{sec:clarif1}--\ref{E:sumN} provide a closed, log–free globalization of the block \(hh\!\to\!h\) with output \(P_N\), while \S\ref{sec:clarif6} together with \S\ref{subsec:global-null}, \S\ref{subsec:patching-local} fix that the “narrow corona” module is applied only in scenarios with “coronal” output at scale \(\sim N^{1-\delta}\).

% --- Bibliography ---
\newpage

\cleardoublepage
\phantomsection
\addcontentsline{toc}{section}{References}
\bibliographystyle{unsrt}
\bibliography{commutator_refs}
\end{document}